\documentclass[12pt,a4paper]{article}

\usepackage{latexsym}
\usepackage{amsmath}
\usepackage{amssymb}
\usepackage{arydshln}

\usepackage{cite}
\usepackage{stmaryrd}
\usepackage{enumerate}

\usepackage{hyperref}

\usepackage{color}
\usepackage{lineno}
\usepackage{graphicx}
\usepackage{ae}
\usepackage{amsmath}
\usepackage{amssymb}
\usepackage{latexsym}
\usepackage{url}
\usepackage{epsfig}
\usepackage{mathrsfs}
\usepackage{amsfonts}
\usepackage{amsthm}

\usepackage{float}
\usepackage{subfig}
\usepackage{tikz}
\usetikzlibrary{positioning}

\newtheorem{theorem}{Theorem}[section]

\newtheorem{prop}[theorem]{Proposition}
\newtheorem{lemma}[theorem]{Lemma}
\newtheorem{remark}{Remark}
\newtheorem{cor}[theorem]{Corollary}

\theoremstyle{definition}

\newtheorem{claim}{Claim}

\numberwithin{equation}{section} 
\allowdisplaybreaks    

\def\qed{\hfill$\Box$\vspace{12pt}}

\long\def\delete#1{}

\usetikzlibrary{decorations.markings}

\tikzstyle{vertex}=[circle, draw, inner sep=0pt, minimum size=6pt]
\tikzstyle{directed}=[postaction={decorate,
	decoration={markings,mark=at position 0.3 with {\arrow{stealth}}}
}]

\usepackage{xcolor}
\usepackage[normalem]{ulem}

\begin{document}
\title {Tur{\'a}n-type extremal problems for unbalanced signed graphs}

\author{Linfeng Xie$^{a,b}$,~Keheng Zhu$^{c}$,~Xiaogang Liu$^{a,b,}$\thanks{Supported by the National Natural Science Foundation of China (No. 12371358).}~$^,$\thanks{ Corresponding author. Email addresses: xielinfeng@mail.nwpu.edu.cn, 2240502168@cnu.edu.cn, xiaogliu@nwpu.edu.cn}~
	\\[2mm]
	{\small $^a$School of Mathematics and Statistics,}\\[-0.8ex]
	{\small Northwestern Polytechnical University, Xi'an, Shaanxi 710072, P.R.~China}\\
	{\small $^b$Xi'an-Budapest Joint Research Center for Combinatorics,}\\[-0.8ex]
	{\small Northwestern Polytechnical University, Xi'an, Shaanxi 710129, P.R. China}\\
	{\small $^c$School of Mathematical Sciences \& Academy for Multidisciplinary Studies,}\\[-0.8ex]
	{\small Captial Normal University, Beijing, 100048, P.R.~China}\\
}
\date{}

\openup 0.5\jot
\maketitle

\begin{abstract}
	In this paper, we establish two Tur{\'a}n-type results for signed graphs. We first generalize the classical Tur{\'a}n theorem to signed graphs and then extend Nikiforov's spectral Tur{\'a}n theorem to signed graphs. Moreover, we determine the second maximum spectral radius among all unbalanced signed graphs that contain no balanced complete signed subgraph on \(r+1\) vertices.
	\smallskip
	
	
\end{abstract}

\section{Introduction}
A \emph{signed graph} $\Gamma=(G,\sigma)$ consists of a simple graph $G=(V(G),E(G))$ together with a signature $\sigma \colon E(G) \to \{+1,-1\}$ on its edge set. The unsigned graph $G$ is called the \emph{underlying graph} of $\Gamma$. In 1946, Heider \cite{Hei-JP-1946} proposed balance theory, which laid the conceptual foundation for the development of signed graphs. Building on Heider's work, Harary \cite{Har-MichM-1953} first formally established signed graph theory, aiming to mathematically formalize balance theory governing group relationships. After nearly three decades, Chaiken \cite{Chaiken-SIAM-1982} and Zaslavsky \cite{Zaslavsky-DAM-1982} independently obtained the matrix-tree theorem for signed graphs. For more information on signed graphs, please refer to \cite{Beck-Zaslavsky-JCTB-2006,Zaslavsky-book-2010,Zaslavsky-EJC-2018,Zaslavsky-JCTB-1987}.

Let $\Gamma=(G,\sigma)$ be a signed graph with the vertex set $V(\Gamma)=\left \{v_{1},v_{2},\dots ,v_{n}\right \}$ and the edge set $E(\Gamma)=\left \{e_{1},e_{2},\dots ,e_{m}\right \}$. Denote by $|V(\Gamma)|$ and $e(\Gamma)$ the order and the size of $\Gamma$, respectively. Let $N_{\Gamma}(v_i)$ denote the set of neighbors of a vertex $v_i$ in $\Gamma$ and $d_{\Gamma}(v_i)=|N_{\Gamma}(v_i)|$ the degree of $v_i$ in $\Gamma$. An edge $v_i v_j$ is called a \emph{positive edge} (respectively, \emph{negative edge}) if $\sigma(v_i v_j)=+1$ (respectively, $\sigma(v_i v_j)=-1$). The set of all positive edges (respectively, negative edges) in $\Gamma$ is denoted by $E^{+}(\Gamma)$ (respectively, $E^{-}(\Gamma)$).  The \emph{adjacency matrix} of $\Gamma$ is defined as $A(\Gamma)=(a_{ij}^{\sigma})$, where $a_{ij}^{\sigma}=\sigma(v_{i}v_{j})a_{ij}$ and $a_{ij}=1$ if $ v_{i} $ and $v_{j}$ are adjacent, and $a_{ij}=0$ otherwise. The eigenvalues of $A(\Gamma)$ are denoted by
$$
\lambda_{1}(\Gamma)\ge \lambda_{2}(\Gamma)\ge \dots \ge \lambda_{n}(\Gamma),
$$
which are called the \emph{adjacency spectrum} of $\Gamma$. The \emph{index} of $\Gamma$ is $\lambda_{1}(\Gamma)$. The \emph{spectral radius} of $\Gamma$ is defined as
$$
\rho(\Gamma)=\max\left \{\lambda_{1}(\Gamma),-\lambda_{n}(\Gamma)\right \}.
$$

Let $\emptyset \ne U\subseteq V(G)$. The operation of changing the signs of all edges between $U$ and $V(G) \setminus U$ is called a \emph{switching}, and is also referred to as \emph{switching $\Gamma$ at a vertex subset $U$}. A signed graph $\Gamma^{\prime}$ is said to be \emph{switching equivalent} to $\Gamma$ if $ \Gamma^{\prime}$ is obtained from $\Gamma$ by a finite sequence of switchings, denoted as $\Gamma\sim \Gamma^{\prime}$; otherwise, we write \(\Gamma\not\sim\Gamma'\). A cycle is called \emph{positive} if the number of its negative edges is even; otherwise, \emph{negative}. A signed graph is called \emph{balanced} if each of its cycles is positive; otherwise, \emph{unbalanced}.

For a given  graph $H$, a graph $G$ is called \emph{$H$-free} if $G$ contains no subgraph isomorphic to $H$. The \emph{Tur{\'a}n number} of $H$, denoted by $ex(n,H)$, is the maximum number of edges an $H$-free graph on $n$ vertices can have. Let $T_r(n)$ denote the Tur{\'a}n graph on $n$ vertices, that is,  the complete $r$-partite graph whose vertex set is partitioned into $r$ parts of almost equal size (each part has either $\lfloor n/r \rfloor$ or $\lceil n/r \rceil$ vertices), and every pair of vertices from different parts is adjacent. Let $K_{n}$ denote the complete graph with $n$ vertices. Mantel's theorem \cite{Mantel-1907} shows that $ex(n, K_3)=\lfloor \frac{n^2}{4} \rfloor$, and the only extremal graph is the $T_2(n)$. Tur{\'a}n's theorem \cite{Turan-MFL-1941}, regarded as the origin of extremal graph theory, determines $ex(n, K_{r+1})$.

\begin{theorem}\emph{(See \cite{Turan-MFL-1941})}\label{Turan-MFL-1941}
	Let $G$ be a graph with $n$ vertices. If $G$ is $K_{r+1}$-free, then
	\[
	e(G)\le e(T_{r}(n)),
	\]
	with equality holding if and only if $G\cong T_{r}(n)$.
\end{theorem}

The spectral Tur{\'a}n theorem can be traced back to 1970, when Nosal \cite{Nosal-Phd-1970} proved that if $G$ is a \(K_3\)-free graph with $m$ edges, then $\lambda_1(G)\le \sqrt{m}$, with equality for complete bipartite graphs. In particular, using Rayleigh's inequality, we have $\frac{2m}{n}\le \lambda_1(G) \le \sqrt{m}$, which implies $m \le \lfloor \frac{n^2}{4} \rfloor$. Thus, Nosal's theorem can be viewed as  a spectral version of Mantel's theorem. In 2007, Nikiforov \cite{Nikiforov-LAA-2007} proved the spectral version of Tur{\'a}n's theorem.

\begin{theorem}\emph{(See \cite[Theorem~1]{Nikiforov-LAA-2007})}\label{Nikiforov-LAA-2007}
	Let $G$ be a graph with $n$ vertices. If $G$ is $K_{r+1}$-free, then
	\[
	\lambda_1(G)\le \lambda_{1}(T_{r}(n)),
	\]
	with equality holding if and only if $G\cong T_{r}(n)$.
\end{theorem}

Since $\frac{2m}{n}\le \lambda_1(G)$ and $e(T_{r}(n))=\left\lfloor \frac{n}{2}\lambda_{1}(T_{r}(n)) \right\rfloor$, we have
\[
e(G)\le \left\lfloor \frac{n}{2}\lambda_1(G)\right\rfloor \le \left\lfloor \frac{n}{2}\lambda_{1}(T_{r}(n)) \right\rfloor=e(T_{r}(n)).
\]
Thus, the spectral bound from Theorem \ref{Nikiforov-LAA-2007} implies the edge bound given by Theorem \ref{Turan-MFL-1941}. In 2009, Nikiforov \cite{Nikiforov-JCTB-2009} asked a question: Is it true that if
$\lambda_1(G)<\lambda_{1}(T_{r}(n))$, then $e(G)< e(T_{r}(n))$? Most recently, Liu and Ning \cite{LB-ar-2026}  answered this question positively in a stronger form. For more details on the aforementioned two Tur{\'a}n problems in unsigned graphs, we refer the readers to \cite{Ni-Wang-Kang-EJC-2023, Yuan-Wang-Zhai-EJLA-2012, Chen-Lei-Li-EJC-2025, Nikiforov-LAA-2017,Bollobas-Nikiforov-JCTB-2007,Dou-Ning-Peng-AAM-2025,Erods-Gallai-1959,Fang-Zhu-Chen-EJC-2025}.

For a given family of signed graphs $\mathcal{F}$, a signed graph $\Gamma$ is called \emph{$\mathcal{F}$-free} if $\Gamma$ contains no subgraph isomorphic to any member of $\mathcal{F}$. Let $\mathcal{K}_{r+1}^+$ (respectively, $\mathcal{K}_{r+1}^{-}$) be the sets of all balanced (respectively, unbalanced) signed graphs whose underlying graph is $K_{r+1}$. Let $(G, +)$ denote the signed graph with all positive edges, whose underlying graph is $G$. Let $T_{r}^{\star}(n)$ (respectively, $T_{r}^\diamond(n)$) be the signed graph obtained by adding a negative edge within one partition class of size \(\lceil n/r \rceil\) (respectively, $\lfloor n/r \rfloor$) of \((T_r(n),+)\).

In this paper, we first generalize Theorem \ref{Turan-MFL-1941} to signed graphs, as follows.

\begin{theorem}\label{signed Turan theorem}
	Let $\Gamma$ be an unbalanced signed graph with $n$ vertices, and define
	$$
	S(\Gamma)=\max_{\Gamma^{\prime} \sim \Gamma} \left(|E^+(\Gamma')|-|E^-(\Gamma')|\right).
	$$
	If $\Gamma$ is $\mathcal{K}_{r+1}^+$-free and $r\ge 2$, then
	\[
	S(\Gamma)\le S(T_{r}^{\star}(n)),
	\]
	with equality holding if and only if $\Gamma\sim T_{r}^{\star}(n)$ or \(\Gamma \sim T_{r}^\diamond(n)\).
\end{theorem}


\begin{figure}[H]
	\hspace{-1.5cm} 
	\begin{minipage}[c]{0.6\textwidth}
		\hspace{2.5cm}
		\begin{tikzpicture}[scale=1.0]
			
			\draw (2,0) ellipse (1.5 and 3);
			
			\node at (2,3.6) {$K_{n-2}$};
			\node at (2,-3.6) {$\Gamma^{1,n-3}$};
			
			\node[circle,fill,inner sep=1.5pt,label=left:$u$] (u) at (-0.6,1) {};
			\node[circle,fill,inner sep=1.5pt,label=left:$v$] (v) at (-0.6,-1) {};
			\node[circle,fill,inner sep=1.5pt,label=right:$u_1$] (u1) at (2,2.3) {};
			\node[circle,fill,inner sep=1.5pt,label=right:$w_1$] (w1) at (2,1.5) {};
			\node[circle,fill,inner sep=1.5pt,label=right:$w_2$] (w2) at (2,0.7) {};
			\node at (2,0) {$\vdots$};
			
			\node at (2,-0.95) {$\vdots$};
			\node[circle,fill,inner sep=1.5pt,label=right:$w_{n-3}$] (w_{n-3}) at (2,-1.7) {}; 
			
			\draw[line width=1.5pt][dashed] (v) -- (u);
			\draw [line width=1.5pt](u) -- (u1);
			\draw [line width=1.5pt](v) -- (w1);
			\draw [line width=1.5pt](v) -- (w2);
			\draw [line width=1.5pt](v) -- (w_{n-3});
		\end{tikzpicture}
	\end{minipage}
	\hfill
	\begin{minipage}[c]{1\textwidth}
		\begin{tikzpicture}[scale =1]
			\draw (2,0) ellipse (1.5 and 3);
			
			\node at (2,3.6) {$K_{n-1}$};
			\node at (2,-3.6) {$\Gamma_{s,n}$};
			
			\node[circle,fill,inner sep=1.5pt,label=left:$v_n$] (v_n) at (-0.6,0) {};
			
			\node[circle,fill,inner sep=1.5pt,label=right:$v_1$] (v1) at (2,2.3) {};
			\node[circle,fill,inner sep=1.5pt,label=right:$v_2$] (v2) at (2,1.6) {};
			\node at (2,1.1) {$\vdots$};
			\node[circle,fill,inner sep=1.5pt,label=right:$v_{s+1}$] (vtm) at (2,0.4) {};
			\node[circle,fill,inner sep=1.5pt,label=right:$v_{s+2}$] (vt) at (2,-0.3) {};
			\node at (2,-0.95) {$\vdots$};
			\node[circle,fill,inner sep=1.5pt,label=right:$v_{n-1}$] (v_{n-1}) at (2,-1.7) {}; 
			
			\draw[line width=1.5pt][dashed] (v_n) -- (v1);
			\draw [line width=1.5pt](v_n) -- (v2);
			\draw [line width=1.5pt](v_n) -- (vtm);
		\end{tikzpicture}
		\hspace{2cm}
	\end{minipage}
	\caption{The signed graphs $\Gamma^{1,n-3}$ and $\Gamma_{s,n}$.}
	\label{figure-1} 	
\end{figure}
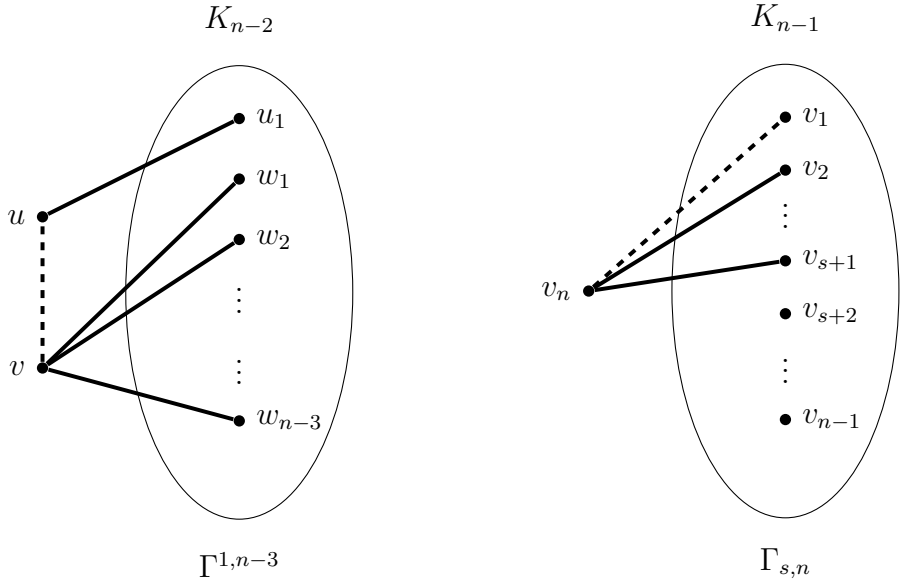

See Fig. \ref{figure-1}. Let $\Gamma^{1,n-3}$ be the signed graph obtained from the all-positive clique $(K_{n-2},+)$ with $V(K_{n-2})=\{u_1,w_1,\dots,w_{n-3}\}$, along with two isolated vertices $u$ and $v$, by adding a negative edge $uv$ and positive edges $\{uu_1,vw_1,\dots,vw_{n-3}\}$. Let  $\Gamma_{s,n}$ be the signed graph obtained by adding a new vertex $v_n$ to $(K_{n-1},+)$ with $V(K_{n-1})=\left\{v_1,v_2,\dots,v_{n-1}\right\}$ and joining $v_n$ to the $s+1$ vertices $v_1,v_2,\dots,v_{s+1}$, where $v_n v_1$ is the unique negative edge.

There are analogous spectral Tur{\'a}n results on signed graphs. Let $\Gamma$ be an unbalanced signed graph with order $n$. In 2024, Wang and Hou \cite{Wang-Hou-Li-LAA-2024} proved that if $\Gamma$ is $\mathcal{K}_{3}^-$-free, then
\[
\rho(\Gamma)\le \frac{1}{2}\sqrt{n^{2}-8}+n-4,
\]
with equality holding if and only if $\Gamma \sim \Gamma^{1,n-3}$. Also in 2024, Chen and Yuan \cite{Chen-Yuan-AMC-2024} proved that if $\Gamma$ is $\mathcal{K}_{4}^-$-free, then
\[
\rho(\Gamma)\le \rho(\Gamma_{1,n}),
\]
with equality holding if and only if $\Gamma \sim \Gamma_{1,n}$. In the same year, Wang \cite{Wang-LAA-2024} proved that if $\Gamma$ is $\mathcal{K}_{5}^-$-free, then
\[
\rho(\Gamma)\le \rho(\Gamma_{2,n}),
\]
with equality holding if and only if $\Gamma \sim \Gamma_{2,n}$. In the same year, Xiong and Hou \cite{Xiong-Hou-AC-2024} proved that if $\Gamma$ is $\mathcal{K}_{s+1}^-$-free with $3\le s\le n-1$, then
\[
\rho(\Gamma)\le \rho(\Gamma_{s-2,n}),
\]
with equality holding if and only if $\Gamma \sim \Gamma_{s-2,n}$. For more details regarding the spectral Tur{\'a}n problem in signed graphs, we refer the readers to to \cite{Li-Qin-DAM-2026,Wang-Lin-DAM-2024,Wang-Hou-Huang-DAM-2025,Xie-Liu-ar-2026-4}.


Secondly, we establish the signed version of Theorem \ref{Nikiforov-LAA-2007}, as follows.

\begin{theorem}\label{Classical spectral Turan for s-g}
	Let $r\ge 2$ and $n\ge r+1$. If $\Gamma$ is an unbalanced $\mathcal{K}^{+}_{r+1}$-free signed graph of order $n$, then
	\[
	\lambda_{1}(\Gamma)\le \lambda_{1}(T_{r}^\star(n)),
	\]
	with equality holding if and only if $\Gamma \sim T_{r}^\star(n)$.
\end{theorem}
For positive integers $n$ and \(r\ge 2\), let \(\Gamma_n(r)\) denote the unbalanced signed complete $r$-partite graph on $n$ vertices, whose partite sizes satisfy \(n_1\ge \cdots\ge n_r\), \(n_1+\cdots+n_r=n\) and \(n_1-n_r\le 1\), and which contains exactly one negative edge \(u_0v_0\), where \(u_0\) lies in the partite set of size \(n_1\) and \(v_0\) lies in the partite set of size \(n_2\).

Finally, we determine the second maximum index among all $\mathcal{K}^{+}_{r+1}$-free unbalanced signed graphs.
\begin{theorem}\label{Classical spectral Turan for s-g-2}
	Let $r\ge 2$ and $n\ge r+1$. If $\Gamma$ is an unbalanced $\mathcal{K}^{+}_{r+1}$-free signed graph of order $n$ and $\Gamma \not \sim T_{r}^\star(n)$, then
	\[
	\lambda_1(\Gamma )\le
	\begin{cases}
		\lambda_1(T_{r}^\diamond(n)), & \text{~if~} s\ne 0, \\[6pt]
		\lambda_1(\Gamma_n(r)),   & \text{~if~}  s=0,
	\end{cases}
	\]
	with equality holding if and only if
	\[
	\Gamma \sim
	\begin{cases}
		T_{r}^\diamond(n), & \text{~if~}s\ne 0, \\[6pt]
		\Gamma_n(r),   &\text{~if~} s=0.
	\end{cases}
	\]
	
\end{theorem}

The rest of the paper is organised as follows. In Section $2$, we prove Theorem \ref{signed Turan theorem}. In Section $3$, we prove Theorems \ref{Classical spectral Turan for s-g} and  \ref{Classical spectral Turan for s-g-2}.  We also present an example to show that it is not true that if $\lambda_{1}(\Gamma)< \lambda_{1}(T_{r}^\star(n))$, then $S(\Gamma)< S(T_{r}^{\star}(n))$.




\section{Proof~of~Theorem~\ref{signed Turan theorem}}
Let $\Gamma=(G,\sigma)$ be a  signed graph with underlying graph $G=(V(G),E(G))$. Let $\emptyset\not=S\subseteq  E(G)$. Denote by $\Gamma-S$ the signed graph obtained from $\Gamma$ by deleting all signed edges of $\Gamma$ whose underlying edges are in $S$. Let $\epsilon(\Gamma)$ be the \emph{frustration index} of $\Gamma$, which is the
minimum number of edges to remove for balance. That is,
$$
\epsilon(\Gamma)=\min_{\Gamma^{\prime} \sim \Gamma} |E^-(\Gamma')|.
$$

\begin{Tproof}\textbf{~of~Theorem~\ref{signed Turan theorem}.}~Note that
	$$
	S(T_{r}^{\star}(n))=\max_{\Gamma^{\prime} \sim T_{r}^{\star}(n)} (|E^+(\Gamma')|-|E^-(\Gamma')|)=\max_{\Gamma^{\prime} \sim T_{r}^{\star}(n)}(e(T_{r}^{\star}(n))-2\epsilon(T_{r}^{\star}(n))).
	$$
	Since $T_{r}^{\star}(n)$ is unbalanced and $|E^-(T_{r}^{\star}(n))|)=1$, we have $\epsilon(T_{r}^{\star}(n)))=1$. Hence,
	\[
	S(T_{r}^{\star}(n))=e(T_{r}^{\star}(n))-2=e(T_r(n))-1.
	\]
	Fix a switching $\Gamma' \sim \Gamma$, and let $S$ be the set of underlying edges corresponding to $E^{-}(\Gamma')$. Then $\Gamma'-S$ is balanced with all positive edges. Since $\Gamma$ is $\mathcal{K}_{r+1}^+$-free, we obtain that $\Gamma'$ is $\mathcal{K}_{r+1}^+$-free and hence $\Gamma'-S$ is $\mathcal{K}_{r+1}^+$-free. By Theorem \ref{Turan-MFL-1941}, we have
	\[
	|E^+(\Gamma')|=e(\Gamma'-S)\le e(T_r(n)).
	\]
	Since $\Gamma$ is unbalanced, the switching-invariance argument also gives $|E^-(\Gamma')|\ge 1$. Therefore,
	\[
	|E^+(\Gamma')|-|E^-(\Gamma')|\le e(T_r(n))-1=S(T_{r}^{\star}(n)).
	\]
	This holds for every switching. Hence, $S(\Gamma)\le S(T_{r}^{\star}(n))$.
	
	Suppose now that $S(\Gamma) =S(T_{r}^{\star}(n))= e(T_r(n))-1$, and choose a switching $\Gamma'$ attaining the maximum. From $|E^+(\Gamma')|\leq e(T_r(n))$, $|E^-(\Gamma')| \geq 1$, and $|E^+(\Gamma')|-|E^-(\Gamma')| = e(T_r(n))-1$, it follows that
	\[
	|E^+(\Gamma')|=e(T_r(n)), ~~|E^-(\Gamma')| = 1.
	\]
	The equality case of Turan's theorem implies that the graph ($E^+(\Gamma')$, $V(\Gamma')$) is $T_r(n)$. Since all pairs in different parts are already positive edges, the unique negative edge must lie inside one part. Hence, $\Gamma'= T_{r}^{\star}(n)~\text{or}~T_{r}^\diamond(n)$.
	
	This completes the proof.\qed
\end{Tproof}

From the proof of Theorem \ref{signed Turan theorem}, we see that Theorem \ref{Turan-MFL-1941} implies Theorem \ref{signed Turan theorem}.
Indeed, Theorem \ref{signed Turan theorem} can also imply Theorem \ref{Turan-MFL-1941} (See Remark \ref{unsigned-signed Turan graph}). Before proceeding, we first present the following result.


\begin{prop}\label{switching equivalence class}
	Let $\Gamma$ be a signed graph such that $\Gamma \sim T_{r}^{\star}(n)$ $($respectively, $T_{r}^\diamond(n))$ and $|E^-(\Gamma)|=1$. Then, for every $n$ and $r\ge 2$, $\Gamma = T_{r}^{\star}(n)$ $($respectively, $T_{r}^\diamond(n))$.
\end{prop}

\begin{proof}
	Let $H$ be the underlying graph of $T_{r}^{\star}(n)$ and $E^{-}(T_{r}^{\star}(n)) = \{e_0\}$ where $e_0=uv$. We assume $u,v\in V_1\subseteq V(H)$ and then $|V_1|=\lceil \frac{n}{r} \rceil$.
	Since $\Gamma \sim T_{r}^{\star}(n)$, there exists a switching set $U \subseteq V(H)$ such that $\Gamma$ is obtained from $T_{r}^{\star}(n)$ by reversing the signs of  the edges in the set
	\[
	\delta_H(U) = \{xy \in E(H) : |\{x, y\} \cap U| = 1\}.
	\]
	Hence, the set of negative edges after switching is
	\begin{equation}\label{switching equivalence-1}
		E^{-}(\Gamma) = \{e_0\} \triangle \delta_H(U),
	\end{equation}
	where $\triangle$ denotes symmetric difference. Since $|E^{-}(\Gamma)| = 1$, we have 	\begin{equation}\label{switching equivalence-2}
		1 = 1 + |\delta_H(U)| - 2\mathbf{1}_{\{e_0\in\delta_H(U)\}},
	\end{equation}
	where
	$$
	\mathbf{1}_{\{e_0\in\delta_H(U)\}}=\begin{cases}
		1, & \text{if~} e_0\in\delta_H(U),\\
		0,& \text{otherwise}.
	\end{cases}
	$$
	
	Suppose that $e_0 \notin \delta_H(U)$. Equation \eqref{switching equivalence-2} yields $|\delta_H(U)| = 0$. Since $H$ is connected, we have $U = \emptyset$ or $U = V(H)$. Every switching preserves the sign of each edge. Hence,
	\[
	\Gamma = T_{r}^{\star}(n).
	\]
	We now consider $e_0 \in \delta_H(U)$. Then, by \eqref{switching equivalence-2}, we obtain
	\begin{equation}\label{switching equivalence-3}
		|\delta_H(U)| = 2.
	\end{equation}
	Since $e_0 = uv \in \delta_H(U)$, without loss of generality, we assume
	\[
	u \in U,\ v \notin U.
	\]
	Clearly, every vertex $w \in V(H) \setminus V_1$ is adjacent to both $u$ and $v$. If $w \in U$, then $vw \in \delta_H(U)$; if $w \notin U$, then $uw \in \delta_H(U)$. Therefore,
	\begin{equation}\label{switching equivalence-4}
		|\delta_H(U)| \geq 1 + |V(H) \setminus V_1| = 1 + n - \left\lceil \frac{n}{r} \right\rceil.
	\end{equation}
	For $n \geq 4$ and $r \geq 2$,
	\[
	\left\lceil \frac{n}{r} \right\rceil \leq \left\lceil \frac{n}{2} \right\rceil \leq n - 2.
	\]
	Hence, \eqref{switching equivalence-4} implies $|\delta_H(U)| \geq 3$, contradicting \eqref{switching equivalence-3}.
	
	Finally, suppose that $n = 3$. Then $r=2$ and the underlying graph of
	$T_{2}^{\star}(3)$ is $K_3$. It is easy to check that
	\[
	\Gamma = T_{2}^{\star}(3).
	\]
	This completes the proof.\qed
\end{proof}

\begin{remark}\label{unsigned-signed Turan graph}
	{\em
		
		Suppose that $G$ is an $n$-vertex \(K_{r+1}\)-free graph with the maximum number of edges. Clearly, $G$ is connected; otherwise, we may add an edge between two connected components and get a \(K_{r+1}\)-free graph with a larger number of edges. This contradicts the maximality of the number of edges of $G$. If $G$ is a complete graph, then Theorem \ref{Turan-MFL-1941} holds trivially. If $G$ is not a complete graph, then there exists at least one vertex \(v\in V(G)\) such that \(G[N(v)]\) is not a complete graph. Thus, $|V(G[N(v)])|\ge 2$. Otherwise, since $G$ is connected, we have $|V(G[N(v)])|= 1$. Set $V(G[N(v)])=\left\{u\right\}$. We can add at least one edge between $v$ and $V\setminus  (V(G[N(u)])\cup \left\{u\right\}) $ to obtain a graph with more edges, a contradiction. In this point, we construct an unbalanced signed graph $\Gamma=(G+e_1,\sigma)$ such that $e_1\in \overline{G}[N(v)]$ and
		\[
		\sigma(e)=
		\begin{cases}
			1, & \text{~if~}e\in E(G),\\
			-1,& \text{~if~}e=e_1.
			
		\end{cases}
		\]
		Recall that $G$ is \(K_{r+1}\)-free. Then $\Gamma$ is $\mathcal{K}_{r+1}^+$-free.
		By Theorem \ref{signed Turan theorem}, we have
		\[
		S(\Gamma)\le S(T_{r}^{\star}(n)).
		\]
		Then
		\[
		e(G)-1=S(\Gamma)\le S(T_{r}^{\star}(n))=e(T_r(n))-1.
		\]
		This implies $e(G)\le e(T_r(n))$.
		
		If $e(G)= e(T_r(n))$, then $S(\Gamma)= S(T_{r}^{\star}(n))$. By Theorem \ref{signed Turan theorem}, we have $\Gamma\sim T_{r}^{\star}(n)$ or \(\Gamma \sim T_{r}^\diamond(n)\). Since $|E^-(\Gamma')| = 1$, by Proposition \ref{switching equivalence class}, we have $\Gamma=T_{r}^{\star}(n)$ or \(\Gamma = T_{r}^\diamond(n)\). For each case, we have $G\cong T_r(n)$.
	}
\end{remark}

\section{Proofs~of~Theorems~\ref{Classical spectral Turan for s-g}~and~\ref{Classical spectral Turan for s-g-2}}

\subsection{Preliminaries}

In this section, we present some necessary results for the proofs of Theorems \ref{Classical spectral Turan for s-g} and \ref{Classical spectral Turan for s-g-2}.
\begin{lemma}\emph{(See \cite[Lemma~2.5]{Sun-Liu-Lan-LAA-2022})}\label{Sun-Liu-Lan-LAA-2022}
	Let $\Gamma$ be a signed graph of order $n$. Then there exists a signed graph $\Gamma^{\prime}$ switching equivalent to $\Gamma$ for which the largest eigenvalue $\lambda_{1}(\Gamma^{\prime})$ admits a non-negative eigenvector.
\end{lemma}

The following two lemmas show that both the balance property and the eigenvalues of signed graphs are invariant under switching equivalence.
\begin{lemma}\emph{(See \cite[Proposition~3.2]{Zaslavsky-DAM-1982})}\label{Zaslavsky-DAM-1982}
	Two signed graphs on the same underlying graph are switching equivalent if and only if they have the same list of balanced cycles.
\end{lemma}

\begin{lemma}\emph{(See \cite[Lemma~2.1]{Hou-Tang-Wang-AMC-2019})}
	\label{Hou-Tang-Wang-AMC-2019}
	Let $\Gamma_{1}=(G,\sigma_{1})$ and $\Gamma_{2}=(G,\sigma_{2})$ be two signed graphs on the same underlying graph $G$. The following are equivalent:
	
	$\mathrm{(1)}$~$\Gamma_{1}$ and $\Gamma_{2}$ are switching equivalent.
	
	$\mathrm{(2)}$~$A(\Gamma_{1})$ and $A(\Gamma_{2})$ are similar.
\end{lemma}

\begin{lemma}\emph{(See \cite[Proposition~5]{Wang-Yan-Qian-LAA-2021})}\label{Wang-Yan-Qian-LAA-2021}
	Let $\Gamma$ be a signed graph of order $n$. Then
	\[\lambda_1(\Gamma)\le n\left(1-\frac{1}{\omega_{b}(\Gamma)}\right).\]
\end{lemma}

\begin{lemma}\emph{(See \cite[p.74]{Cvetkovic-Doob-Sachs-Book-1980})}\label{Cvetkovic-Doob-Sachs-Book-1980}
	Let $K_{n_1,n_2,\dots,n_r}$ be the complete $r$-partite graph whose $r$ color classes have sizes $n_1,n_2,\dots,n_r$, respectively. Then the largest eigenvalue $\lambda$ of $K_{n_1,n_2,\dots,n_r}$ satisfies
	\[	
	\sum_{i=1}^r \frac{n_i}{\lambda + n_i} = 1.
	\]
\end{lemma}

\begin{lemma}\emph{(See \cite[Theorem~1.2]{Cai-Zhou-ar-2026})}\label{Cai-Zhou-ar-2026}
	Let $\Gamma$ be an unbalanced $r$-partite signed graph of order $n$ with $r\ge 2$. Then
	\[
	\lambda_1(\Gamma)\le \lambda_1\big(\Gamma_n(r)\big),
	\]
	with equality if and only if $\Gamma\sim \Gamma_n(r)$.
\end{lemma}

\noindent\textbf{Notation}: Throughout the rest paper, we fix integers \(r\ge 2\), and \(n,q,s\) satisfying the division decomposition
\(n=rq+s,~0\le s<r.\) For brevity, we introduce the shorthand notation
\[
L=\left(1-\frac{1}{r}\right)n,
\]
which will be used consistently.

Suppose that the Tur{\'a}n graph $T_r(n)$ has vertex partition $V_1,V_2,\dots,V_r$ with $|V_i|=n_i$. Let $\Theta$ denote the signed graph obtained by adding a negative edge $uv$ inside the partition class $V_j$, where $|V_j|=k$, to $T_{r}(n)$. Let
$$
h_k(t)=\frac{k}{t+k},\text{~~and~~}g_k(t)=\frac{k - \dfrac{2}{t+1}}{t + k - \dfrac{2}{t+1}}.
$$
Define
$$\Phi_T(t)=\sum_{i=1}^r \frac{n_i}{t + n_i},$$
and
\[
\Phi_k(t) = \Phi_T(t) - h_k(t) + g_k(t).
\]

\begin{lemma}\label{eigenvalue of T_{r}(n)}
	$\lambda_1(T_{r}(n))$ is the unique solution of the equation
	\[
	\Phi_T(t)=1
	\]
	in the region $t>0$. Moreover, $\lambda_1(T_{r}(n))\le L$ with equality holding if and only if $s=0$.
\end{lemma}

\begin{proof}
	For $T_{r}(n)$, by Lemma \ref{Cvetkovic-Doob-Sachs-Book-1980}, we have
	\[
	\frac{(r-s)q}{\lambda_1(T_{r}(n))+q} + \frac{s(q+1)}{\lambda_1(T_{r}(n))+q+1}=1.
	\]
	Hence, $\lambda_1(T_{r}(n))$ is the unique solution of the equation
	\[
	t^2-(n-2q-1)t-(r-1)q(q+1)=0
	\]
	in the region $t>0$. Clearly,
	\[
	\lambda_1(T_{r}(n))=\frac{n - 2q - 1 + \sqrt{(n - 2q - 1)^2 + 4(r - 1)q(q + 1)}}{2}.
	\]
	Let $f(t)=t^2-(n-2q-1)t-(r-1)q(q+1)$. We have
	\[
	f(L)=\frac{(r-1)s(r-s)}{r^2}\ge 0.
	\]
	Hence, $\lambda_1(T_{r}(n))\le L$, with equality if and only if \(s=0\).\qed
\end{proof}


\begin{lemma}\label{eigenvalue of A(Sigma)}
	For any $t>1$, $t$ is an eigenvalue of $A(\Theta)$ if and only if
	\[
	\Phi_k(t) = 1.
	\]
	Moreover, $\Phi_k$ is strictly decreasing on $t>1$, and $\lim_{t\to\infty}\Phi_k(t) = 0$.
\end{lemma}
\begin{proof}
	If $t$ is an eigenvalue of $A(\Theta)$, let $\mathbf{x}$ be an eigenvector corresponding to the eigenvalue $t>1$. By symmetry, $x_u=x_v$. For any $x_l\in V_j \setminus \left\{u,v\right\}$, we have
	\[
	tx_l=\sum_{w\in V(\Sigma^{\pm})\setminus V_j}x_w,~~tx_u=-x_v+\sum_{w\in V(\Sigma^{\pm})\setminus V_j}x_w.
	\]
	Hence, $x_u=\frac{t}{t+1}x_l$. Let $S=\sum_{v\in V(\Sigma^{\pm})}x_v$. We have
	\[
	S=\sum_{w\in V(\Sigma^{\pm})\setminus V_j}x_w+2x_u+(k-2)x_l=(t+k-\frac{2}{t+1})x_l.
	\]
	For any partition $V_i\ne V_j$ and $x_i\in V_i$, we have $tx_i=S-n_ix_i$ and then
	$x_i=\frac{S}{t+n_i}$.
	By direct calculation,
	\[
	2x_u+(k-2)x_l=(k-\frac{2}{t+1})x_l=Sg_k(t).
	\]
	Then,
	\[
	\sum_{i\ne j}S\frac{n_i}{t+n_i}+Sg_k(t)=S.
	\]
	By eliminating $S$, we obtain
	\[
	\Phi_k(t) = 1.
	\]
	
	Conversely, suppose that $t>1$ satisfies $\Phi_m(t)=1$. Set $S=1$, and define
	\[
	x_l = \frac{1}{t + m - \displaystyle\frac{2}{t+1}},\quad
	x_u = \frac{t}{t+1}x_l,\quad
	x_i = \frac{1}{t + n_i} \quad (V_i \neq V_j).
	\]
	The equation $\Phi_k(t)=1$ guarantees that the sum of these coordinates equals $1$. Substituting them pointwise back into the above characteristic equation yields a nonzero eigenvector, so $t$ is an eigenvalue.
	
	Finally, $h_k'(t) < 0$. By simple calculation, we have
	\[
	g_k'(t) = \frac{\dfrac{4t+2}{(t+1)^2} - k}{\left(t + k - \dfrac{2}{t+1}\right)^2} < 0
	\qquad (t>1,\ k\ge 2),
	\]
	since $k\ge 2$ and
	\[
	2 - \frac{4t+2}{(t+1)^2} = \frac{2t^2}{(t+1)^2} > 0.
	\]
	Thus $\Phi_k$ is strictly decreasing and clearly tends to $0$. \qed
\end{proof}

Define
\[
\Phi_+(t) =
\begin{cases}
	\displaystyle g_q(t) + \frac{(r-1)q}{t+q}, & s=0,\\[0.4cm]
	\displaystyle \frac{(r-s)q}{t+q} + \frac{(s-1)(q+1)}{t+q+1} + g_{q+1}(t), & 1\le s < r.
\end{cases}
\]

\begin{lemma}\label{eigenvalue of T_{r}^{+}(n)}
	If $(r,q,s)=(2,1,1)$, then $\lambda_1(T_{r}^\star(n))=1$. Otherwise, $\lambda_1(T_{r}^\star(n))$ is the unique solution of the equation
	\[
	\Phi_+(t)=1
	\]
	in the region $t>1$. Equivalently, for all admissible cases, $\lambda_1(T_{r}^\star(n))$ is the unique solution of $\Phi_+(t)=1$ with $t\geq 1$.
\end{lemma}
\begin{proof}
	We apply Lemma \ref{eigenvalue of A(Sigma)} to the largest partition of $\Theta$. If $s=0$, then $k=q$, and the remaining $r-1$ partitions are all of size $q$. If $s>0$, then $k=q+1$, and the other partitions consist of $r-s$ copies of $q$ and $s-1$ copies of $q+1$. Thus Lemma \ref{eigenvalue of A(Sigma)} exactly yields the aforementioned function $\Phi_+$.
	
	When $s=0$ and $q\ge 2$,
	\[
	\Phi_+(1) = \frac{q-1}{q} + \frac{(r-1)q}{q+1} > 1.
	\]
	
	When $s>0$ and $q\ge 1$,
	\[
	\Phi_+(1) = \frac{(r-s+1)q}{q+1} + \frac{(s-1)(q+1)}{q+2}.
	\]
	If $s=1$, this expression equals $\dfrac{rq}{q+1}$, which equals $1$ only for $(r,q)=(2,1)$ and is greater than $1$ otherwise. If $s\ge 2$, then
	\[
	\Phi_+(1) \ge \frac{q}{q+1} + \frac{q+1}{q+2} > 1.
	\]
	By Lemma \ref{eigenvalue of A(Sigma)}, $\Phi_+$ is strictly decreasing on $t>1$.
	Then, the equation $\Phi_+(t)=1$ has a unique solution with $t>1$ except for the exceptional case $(r,q,s)=(2,1,1)$. In this exceptional case, $\Phi_+(1)=1$ and there is no solution for $t>1$. Every eigenvalue greater than $1$ is given by the root of this equation, this completes the proof.\qed
\end{proof}

Define
\[
\Phi_-(t) =\frac{(r-s-1)q}{t+q} + \frac{s(q+1)}{t+q+1} + g_{q}(t), ~1\le s < r.
\]
When $s=0$, $\lambda_1(T_{r}^\star(n))=\lambda_1(T_{r}^\diamond(n))$. Similar to the proof of Lemma \ref{eigenvalue of T_{r}^{+}(n)}, we obtain the following result immediately.
\begin{cor}\label{eigenvalue T_{r}^{-}(n)}
	If $1\le s<r$ and $q\ge 2$, then $\lambda_1(T_{r}^\diamond(n))$ is the unique solution of the equation
	\[
	\Phi_-(t)=1
	\]
	in the region $t>1$.
\end{cor}

We present the following theorem as a preliminary step toward proving Theorems \ref{Classical spectral Turan for s-g} and \ref{Classical spectral Turan for s-g-2}.
\begin{theorem}\label{eigenvalue-T_{r}^{+-}(n)}
	Let \[
	\gamma =\lceil \frac{n}{r} \rceil=
	\begin{cases}
		q, & s=0, \\
		q+1, & 1\le s<r.
	\end{cases}
	\] If $1\le s<r$ and $q\ge 2$, then
	$$
	\lambda_{1}(T_{r}^\diamond(n))<\lambda_{1}(T_{r}^\star(n)).
	$$
	Moreover, for $0\le s<r$,
	$$
	\lambda_{1}(T_{r}^\star(n))<\lambda_1(T_{r}(n)).
	$$
\end{theorem}

\begin{proof}
	Let
	\[
	\Delta_k(t) = h_k(t) - g_k(t).
	\]
	Direct simplification yields
	\[
	\Delta_k(t) = \frac{2t}{(t+1)(t+k)\left(t + k - \dfrac{2}{t+1}\right)} > 0.
	\]
	When $s=0$, $\lambda_1(T_{r}^\star(n))=\lambda_1(T_{r}^\diamond(n))$. Suppose that $1\le s<r$. Then
	\[
	\Phi_-(t)= \Phi_T(t)-\Delta_q(t),~\Phi_+(t)= \Phi_T(t)-\Delta_{q+1}(t).
	\]
	For fixed $t\ge 1$, $\Delta_k(t)$ is strictly decreasing in $k$. Then, we obtain $\Phi_+(t)>\Phi_-(t)$. By Corollary \ref{eigenvalue T_{r}^{-}(n)}, $\Phi_-(\lambda_1(T_{r}^\diamond(n)))=1$. Then, $\Phi_+(\lambda_1(T_{r}^\diamond(n)))>1$. By Lemmas \ref{eigenvalue of A(Sigma)} and \ref{eigenvalue of T_{r}^{+}(n)}, $\Phi_+$ is strictly decreasing on $t>1$ and $\Phi_+(\lambda_1(T_{r}^\star(n)))=1$. Hence, $\lambda_1(T_{r}^\star(n))>\lambda_1(T_{r}^\diamond(n))$.

	Regardless of whether \(s = 0\), we have
	$$
	\Phi_+(t) =\Phi_T(t)-\Delta_\gamma(t).
	$$
	By Lemma \ref{eigenvalue of T_{r}(n)}, $\Phi_T(\lambda_1(T_{r}(n))) = 1$. Note that $\Delta_\gamma(t) > 0$, so $\Phi_+(\lambda_1(T_{r}(n))) < 1$. Recall that $\Phi_+$ is strictly decreasing on $t>1$ and $\Phi_+(\lambda_1(T_{r}^\star(n))) = 1$. We obtain $\lambda_1(T_{r}^\star(n)) < \lambda_1(T_{r}(n))$.
	
	This completes the proof.\qed
\end{proof}

\subsection{Proof~of~Theorem~\ref{Classical spectral Turan for s-g}}
\begin{Tproof}\textbf{~of~Theorem~\ref{Classical spectral Turan for s-g}.}~Let $\Gamma=(G,\sigma)$ be a signed graph having the maximum index over all $\mathcal{K}^{+}_{r+1}$-free unbalanced signed graph of order $n$. By Lemma \ref{Sun-Liu-Lan-LAA-2022}, let $\Gamma^{\prime}=(G,\sigma^{\prime})$ be the signed graph switching equivalent to $\Gamma$ for which $\lambda_{1}(\Gamma^{\prime})$ admits a non-negative eigenvector. Let $V(\Gamma^{\prime}) = \left \{v_1,v_2,\dots,v_n \right \}$ and $\mathbf{x}=\left (x_1,x_2,\dots,x_n \right )^T$ be the non-negative unit eigenvector of $A(\Gamma^{\prime})$ corresponding to $\lambda_1(\Gamma^{\prime})$. By Lemmas \ref{Zaslavsky-DAM-1982} and \ref{Hou-Tang-Wang-AMC-2019}, $\Gamma^{\prime}=(G,\sigma^{\prime})$ is also $\mathcal{K}^{+}_{r+1}$-free unbalanced signed graph with the maximum index. Note that $T_{r}^\star(n)$ is $\mathcal{K}^{+}_{r+1}$-free unbalanced signed graph. By the Rayleigh principle, we have
	\begin{equation}\label{Index of signed Turan graph-1}
		\lambda_{1}(\Gamma^{\prime})\ge \lambda_1(T_{r}^\star(n))\ge \frac{\mathbf{1}^T A(T_{r}^\star(n))\mathbf{1}}{\mathbf{1}^T \mathbf{1}}=L- \frac{s(r - s)}{rn} - \frac{2}{n}.
	\end{equation}

	\begin{claim}\label{eigenvector of one zero coordinate}
		The eigenvector $\mathbf{x}$ has at most one zero coordinate.
	\end{claim}
	\noindent\emph{Proof of Claim 1.}~If $\mathbf{x}$ has two zero coordinates, say $x_i$ and $x_j$. By Lemma \ref{Wang-Yan-Qian-LAA-2021}, then
	$$
	\lambda_1(\Gamma^{\prime})=\mathbf{x}^{T}A(\Gamma^{\prime})\mathbf{x}\le \lambda_1(\Gamma^{\prime}-v_i-v_j)\le \left(1-\frac{1}{\omega_b(\Gamma)}\right)(n-2)\le \left(1-\frac{1}{r}\right)(n-2).
	$$
	If $q\ge 2$, then $n\ge 2r$, and $s(r-s)\le \frac{r^2}{4}$, so
	\[
	2(r-1)n - s(r-s)-2r \ge \frac{r(15r-24)}{4} > 0.
	\]
	If $q=1$, then $n=r+s$ with $s\ge 1$, and
	\[
	2(r-1)n - s(r-s)-2r = (r-2)(2r+s)+s^2 > 0.
	\]
	Hence,
	\begin{align*}
		&L- \frac{s(r - s)}{rn} - \frac{2}{n}-\left(1-\frac{1}{r}\right)(n-2)
		\\
		&= \frac{2(r-1)n - s(r-s)-2r}{rn}\\
		&>0,
	\end{align*}
	which contradicts equality \eqref{Index of signed Turan graph-1}.
	
	\begin{claim}\label{connected-1}
		$\Gamma^{\prime}$ is connected.
	\end{claim}
	\noindent\emph{Proof of Claim 2.}~By contradiction, assume that $\Gamma^{\prime}_1$ and $\Gamma^{\prime}_2$ are two distinct connected components of $\Gamma^{\prime}$, where $\lambda_1(\Gamma^{\prime})=\lambda_1(\Gamma^{\prime}_1)$. Without loss of generality, suppose that $v_i \in V(\Gamma^{\prime}_1)$ and $v_j \in V(\Gamma^{\prime}_2)$. Let $\Gamma^{*}$ be the signed graph obtained from $\Gamma^{\prime}$ by adding a positive edge $v_i v_j$. Clearly, $\Gamma^{*}$ is unbalanced and $\mathcal{K}^{+}_{r+1}$-free. By the Rayleigh principle, we have
	$$
	\lambda_1(\Gamma^{*})-\lambda_1(\Gamma^{\prime})\ge \mathbf{x}^{T}A(\Gamma^{*})\mathbf{x}-\mathbf{x}^{T}A(\Gamma^{\prime})\mathbf{x}=2x_i x_j\ge 0.
	$$
	If $\lambda_1(\Gamma^{*})=\lambda_1(\Gamma^{\prime})$, then $\mathbf{x}$ is an eigenvector of $A(\Gamma^{*})$ corresponding to $\lambda_1(\Gamma^{*})$. Note that
	$$
	\lambda_1(\Gamma^{\prime})x_i=\sum_{v_k \in N_{\Gamma^{\prime}}(v_i)}\sigma^{\prime}(v_k v_i)x_k, ~~
	\lambda_1(\Gamma^{\prime})x_j=\sum_{v_k \in N_{\Gamma^{\prime}}(v_j)}\sigma^{\prime}(v_k v_j)x_k
	$$
	and
	$$
	\lambda_1(\Gamma^{*})x_i=\sum_{v_k \in N_{\Gamma^{\prime}}(v_i)}\sigma^{\prime}(v_k v_i)x_k+x_j, ~~
	\lambda_1(\Gamma^{*})x_j=\sum_{v_k \in N_{\Gamma^{\prime}}(v_j)}\sigma^{\prime}(v_k v_j)x_k+x_i.
	$$
	Then we have $x_i=x_j=0$, which contradicts  Claim \ref{eigenvector of one zero coordinate}. Hence, $\lambda_1(\Gamma^{*})>\lambda_1(\Gamma^{\prime})$, which contradicts the maximality of $\lambda_1(\Gamma^{\prime})$.

	Since $\Gamma'$ is unbalanced, $\Gamma'$ contains at least one negative edge and at least one negative cycle. Let $\mathcal{L}$ be one of shortest negative cycles of $\Gamma'$.
	\begin{claim}\label{negative edge}
		$\mathcal{L}$ contains all negative edges of $\Gamma'$.
	\end{claim}
	\noindent\emph{Proof of Claim 3.}~If there is a negative edge $v_i v_j \in E(\Gamma')\setminus E(\mathcal{L})$, then we construct a new graph $\Gamma^{*}$ obtained from $\Gamma^{\prime}$ by deleting $v_i v_j$. Obviously, $\Gamma^{*}$ is unbalanced and $\mathcal{K}_{r+1}^{+}$-free. By the Rayleigh principle,
	$$
	\lambda_1(\Gamma^{*})-\lambda_1(\Gamma^{\prime})\ge \mathbf{x}^{T}A(\Gamma^{*})\mathbf{x}-\mathbf{x}^{T}A(\Gamma^{\prime})\mathbf{x}=2x_i x_j\ge 0.
	$$
	If $\lambda_1(\Gamma^{*})=\lambda_1(\Gamma^{\prime})$, then $\mathbf{x}$ is an eigenvector of $A(\Gamma^{*})$ corresponding to $\lambda_1(\Gamma^{*})$. Note that
	$$
	\lambda_1(\Gamma^{\prime})x_i=\sum_{v_k \in N_{\Gamma^{\prime}}(v_i)}\sigma^{\prime}(v_k v_i)x_k, ~~
	\lambda_1(\Gamma^{\prime})x_j=\sum_{v_k \in N_{\Gamma^{\prime}}(v_j)}\sigma^{\prime}(v_k v_j)x_k
	$$
	and
	$$
	\lambda_1(\Gamma^{*})x_i=\sum_{v_k \in N_{\Gamma^{\prime}}(v_i)}\sigma^{\prime}(v_k v_i)x_k+x_j, ~~
	\lambda_1(\Gamma^{*})x_j=\sum_{v_k \in N_{\Gamma^{\prime}}(v_j)}\sigma^{\prime}(v_k v_j)x_k+x_i.
	$$
	Then we have $x_i=x_j=0$, which contradicts Claim \ref{eigenvector of one zero coordinate}. Hence, $\lambda_1(\Gamma^{*})>\lambda_1(\Gamma^{\prime})$, which contradicts the maximality of $\lambda_1(\Gamma^{\prime})$.
	
	\begin{claim}\label{one negative edge-1}
		$\mathcal{L}$ contains exactly one negative edge.
	\end{claim}
	\noindent\emph{Proof of Claim 4.}~Suppose that $\mathcal{L}$ contains at least $l(\ge 3)$ negative edges. We now proceed with a case analysis on $|V(\mathcal{L})|$.
	
	\noindent\emph{Case~1.}~Assume that $|V(\mathcal{L})|\ge 4$. For any negative edge $uv\in E(\mathcal{L})$, $u$ and $v$ have no common neighbors. Otherwise, we assume that $w\in N_{\Gamma'}(u)\cap N_{\Gamma'}(v)$. By Claim \ref{negative edge}, we obtain that $\sigma'(wu)=\sigma'(wv)=1$. Then, $w,u,v$ form an unbalanced triangle, which contradicts the minimality of the shortest negative cycle $\mathcal{L}$. We construct a new signed graph $\Gamma^{*}$ by reversing the signs of two negative edges of $\mathcal{L}$ in $\Gamma'$. Clearly, $\Gamma^{*}$ is unbalanced and $\mathcal{K}_{r+1}^{+}$-free. Suppose the reversed negative edges are $\left\{v_{i_1}v_{j_1},v_{i_2}v_{j_2}\right\}$.
	
	If $v_{i_1}v_{j_1}$ and $v_{i_2}v_{j_2}$ are adjacent, without loss of generality, let $v_{j_1}=v_{j_2}$. By the Rayleigh principle,
	$$
	\lambda_1(\Gamma^{*})-\lambda_1(\Gamma^{\prime})\ge \mathbf{x}^{T}A(\Gamma^{*})\mathbf{x}-\mathbf{x}^{T}A(\Gamma^{\prime})\mathbf{x}= 4x_{j_1} (x_{i_1}+x_{i_2})\ge 0.
	$$
	If $\lambda_1(\Gamma^{*})=\lambda_1(\Gamma^{\prime})$, then $\mathbf{x}$ is an eigenvector of $A(\Gamma^{*})$ corresponding to $\lambda_1(\Gamma^{*})$ and $x_{j_1} (x_{i_1}+x_{i_2})=0$. By Claim \ref{eigenvector of one zero coordinate}, $x_{j_1}=0$. Note that
	$$
	\lambda_1(\Gamma^{\prime})x_{j_1}=\sum_{v_k \in N_{\Gamma^{\prime}}(v_{j_1})}\sigma^{\prime}(v_k v_{j_1})x_k
	$$
	and
	$$
	\lambda_1(\Gamma^{*})x_{j_1}=\sum_{v_k \in  N_{\Gamma^{\prime}}(v_{j_1})}\sigma^{\prime}(v_k v_{j_1})x_k+2x_{i_1}+2x_{i_2}.
	$$
	Then we have $x_{i_1}=x_{i_2}=x_{j_1}=0$, which contradicts Claim \ref{eigenvector of one zero coordinate}. Hence, $\lambda_1(\Gamma^{*})>\lambda_1(\Gamma^{\prime})$, which contradicts the maximality of $\lambda_1(\Gamma^{\prime})$.
	
	If $v_{i_1}v_{j_1}$ and $v_{i_2}v_{j_2}$ are non-adjacent, by the Rayleigh principle, we have
	$$
	\lambda_1(\Gamma^{*})-\lambda_1(\Gamma^{\prime})\ge \mathbf{x}^{T}A(\Gamma^{*})\mathbf{x}-\mathbf{x}^{T}A(\Gamma^{\prime})\mathbf{x}= 4x_{i_1}x_{j_1}+4x_{i_2}x_{j_2}\ge 0.
	$$
	If $\lambda_1(\Gamma^{*})=\lambda_1(\Gamma^{\prime})$, then $\mathbf{x}$ is also an eigenvector of $A(\Gamma^{*})$ corresponding to $\lambda_1(\Gamma^{*})$. Note that
	$$
	\lambda_1(\Gamma^{\prime})x_{i_1}=\sum_{v_k \in N_{\Gamma^{\prime}}(v_{i_1})}\sigma^{\prime}(v_k v_{i_1})x_k, ~~
	\lambda_1(\Gamma^{\prime})x_{j_1}=\sum_{v_k \in N_{\Gamma^{\prime}}(v_{j_1})}\sigma^{\prime}(v_k v_{j_1})x_k
	$$
	and
	$$
	\lambda_1(\Gamma^{*})x_{i_1}=\sum_{v_k \in N_{\Gamma^{\prime}}(v_{i_1})}\sigma^{\prime}(v_k v_{i_1})x_k+2x_{j_1}, ~~
	\lambda_1(\Gamma^{*})x_{j_1}=\sum_{v_k \in N_{\Gamma^{\prime}}(v_{j_1})}\sigma^{\prime}(v_k v_{j_1})x_k+2x_{i_1}.
	$$
	Then we have $x_{i_1}=x_{j_1}=0$, which also contradicts Claim \ref{eigenvector of one zero coordinate}. Hence, $\lambda_1(\Gamma^{*})>\lambda_1(\Gamma^{\prime})$, which contradicts the maximality of $\lambda_1(\Gamma^{\prime})$.
	
	\noindent\emph{Case~2.}~Assume that $|V(\mathcal{L})|= 3$. Suppose that $V(\mathcal{L})=\left\{v_1,v_2,v_3\right\}$. We assert that $\left(N_{\Gamma}(v_1)\cap N_{\Gamma}(v_2)\right)\setminus \left\{v_3\right\}=\emptyset$. Otherwise, there exists a vertex $v_4\in \left(N_{\Gamma}(v_1)\cap N_{\Gamma}(v_2)\right)\setminus \left\{v_3\right\}$. By Claim \ref{negative edge}, $v_1,v_2,v_4$ form an unbalanced triangle with one negative edge. We construct a new signed graph $\Gamma^{''}$ by deleting $v_1v_3$. Clearly, $\Gamma^{''}$ is unbalanced and $\mathcal{K}_{r+1}^{+}$-free. By the Rayleigh principle,
	$$
	\lambda_1(\Gamma^{''})-\lambda_1(\Gamma^{\prime})\ge \mathbf{x}^{T}A(\Gamma^{''})\mathbf{x}-\mathbf{x}^{T}A(\Gamma^{\prime})\mathbf{x}= 2x_{1} x_{3}\ge 0.
	$$
	If $\lambda_1(\Gamma^{''})=\lambda_1(\Gamma^{\prime})$, then $\mathbf{x}$ is an eigenvector of $A(\Gamma^{''})$ corresponding to $\lambda_1(\Gamma^{''})$. Note that
	$$
	\lambda_1(\Gamma')x_{1}=\sum_{v_k \in N_{\Gamma^{\prime}}(v_{1})}\sigma^{\prime}(v_k v_{1})x_k~~
	\lambda_1(\Gamma^{\prime})x_3=\sum_{v_k \in N_{\Gamma^{\prime}}(v_3)}\sigma^{\prime}(v_k v_3)x_k
	$$
	and
	$$
	\lambda_1(\Gamma^{''})x_1=\sum_{v_k \in N_{\Gamma^{\prime}}(v_1)}\sigma^{\prime}(v_k v_1)x_k+x_3, ~~
	\lambda_1(\Gamma^{''})x_3=\sum_{v_k \in N_{\Gamma^{\prime}}(v_3)}\sigma^{\prime}(v_k v_3)x_k+x_1.
	$$
	Then we have $x_{1}=x_{3}=0$, which contradicts Claim \ref{eigenvector of one zero coordinate}. Hence, $\lambda_1(\Gamma^{*})>\lambda_1(\Gamma^{\prime})$, which contradicts the maximality of $\lambda_1(\Gamma^{\prime})$. Then $\left(N_{\Gamma}(v_1)\cap N_{\Gamma}(v_2)\right)\setminus \left\{v_3\right\}=\emptyset$. Similarly, $\left(N_{\Gamma}(v_1)\cap N_{\Gamma}(v_3)\right)\setminus \left\{v_2\right\}=\emptyset$. Let $\Sigma$ be the signed graph obtained from $\Gamma^{\prime}$ by reversing the signs of the negative edges $v_1 v_2$ and $v_1 v_3$. Clearly, $\Sigma$ is also unbalanced and $\mathcal{K}_{r+1}^{+}$-free. Similar to the Case $1$. We have $\lambda_1(\Sigma)>\lambda_1(\Gamma^{\prime})$,  which contradicts the maximality of $\lambda_1(\Gamma^{\prime})$.
	
	By Claims \ref{negative edge} and \ref{one negative edge-1}, without loss of generality, assume $v_1v_2$ is the unique negative edge in $\Gamma'$ and that $x_1\le x_2$. By Claim \ref{eigenvector of one zero coordinate}, $x_2>0$.
	\begin{claim}\label{At most one component of x is zero}
		$x_i >0$ for $3\le i \le n$.
	\end{claim}
	\noindent\emph{Proof of Claim 5.}~~By contradiction, assume that $x_i=0$ for some $i$. Recall that $\Gamma^{\prime}$ contains exactly one negative edge. Then, by Claim \ref{connected-1}, we have $d_{\Gamma^{\prime}}(v_i)\ge 1$ and all edges incident to $v_i$ are positive. By Claim \ref{eigenvector of one zero coordinate}, we have
	$$
	0=\lambda_1(\Gamma^{\prime})x_i=\sum_{v_k \in N_{\Gamma^{\prime}}(v_i)}x_k>0,
	$$
	a contradiction.
	
	For $U\subseteq V(\Gamma)$, let $\Gamma[U]$ denote the signed subgraph of $\Gamma$ induced by $U$, with edge signs inherited from $\Gamma$. Sometimes, we say that $U$ induces $\Gamma[U]$. For any $e\in E(\Gamma)$, $e$ is a \emph{cut edge} if and only if the number of connected components of $\Gamma$ increases upon deleting $e$.
	Recall that $\Gamma'$ contains exactly one negative edge. For all subsequent proofs, any edge addition or deletion without explicit specification refers to positive edges.
	\begin{claim}\label{complete r-partite}
		$\Gamma'[V(\Gamma^{\prime})\setminus \left\{v_1,v_2\right\}]$ is a complete $r$-partite signed graph with all positive edges.
	\end{claim}
	\noindent\emph{Proof of Claim 6.}~Set $\tilde{\Gamma}=\Gamma'[V(\Gamma^{\prime})\setminus \left\{v_1,v_2\right\}]$. We assert that for any two non-adjacent vertices $u,v \in V(\tilde{\Gamma})$, $N_{\tilde{\Gamma}}(u)=N_{\tilde{\Gamma}}(v)$. By contradiction, assume that there are two non-adjacent vertices $u$ and $v$ with $N_{\tilde{\Gamma}}(u)\ne N_{\tilde{\Gamma}}(v)$. Without loss of generality, let $w\in N_{\tilde{\Gamma}}(u)\setminus N_{\tilde{\Gamma}}(v)$.
	
	\noindent\emph{Case~1.}~$
	\sum_{v_k \in N_{\Gamma'}(v)}x_k<\sum_{v_k \in N_{\Gamma'}(u)}x_k$
	or
	$
	\sum_{v_k \in N_{\Gamma'}(v)}x_k<\sum_{v_k \in N_{\Gamma'}(w)}x_k$.
	
	If $\sum_{v_k \in N_{\Gamma'}(v)}x_k<\sum_{v_k \in N_{\Gamma'}(u)}x_k$, let $\Gamma^{*}=\Gamma'-\left\{vi \mid i\in N_{\Gamma'}(v)\right\}+\left\{vi \mid i\in N_{\Gamma'}(u)\right\}$. Obviously, $\Gamma^{*}$ is $\mathcal{K}_{r+1}^{+}$-free. We first assume that $\Gamma^{*}$ is unbalanced. We obtain that
	\begin{align*}
		\lambda_1(\Gamma^{*})-\lambda_1(\Gamma')&\ge \mathbf{x}^{T}A(\Gamma^{*})\mathbf{x}-\mathbf{x}^{T}A(\Gamma')\mathbf{x}\\
		&=2x_v\sum_{v_k \in N_{\Gamma'}(u)}x_k-2x_v\sum_{v_k \in N_{\Gamma'}(v)}x_k\\
		&>0.
	\end{align*}
	This contradicts the maximality of $\lambda_1(\Gamma^{\prime})$. When $\sum_{v_k \in N_{\Gamma'}(v)}x_k<\sum_{v_k \in N_{\Gamma'}(w)}x_k$, the proof is similar. If $\Gamma^{*}$ is balanced, this implies that the negative edge $v_1 v_2$ is a cut edge of $\Gamma^{*}$. Since $\Gamma^{*}$ is connected, suppose that \(v_1v_3\in E(\Gamma')\) or $v_2v_4\in E(\Gamma')$. Without loss of generality, \(v_1v_3\in E(\Gamma')\). Recall that $\mathbf{x}$ has at most one zero coordinate. Then we obtain that
	\[
	\lambda_{1}(\Gamma^{*}+v_2 v_3)>\lambda_{1}(\Gamma^{*})>\lambda_{1}(\Gamma').
	\]
	Since $\Gamma^{*}+v_2 v_3$ is unbalanced and $\mathcal{K}_{r+1}^{+}$-free, this contradicts the maximality of $\lambda_1(\Gamma^{\prime})$.
	
	\noindent\emph{Case~2.}~$
	\sum_{v_k \in N_{\Gamma'}(v)}x_k\ge \sum_{v_k \in N_{\Gamma'}(u)}x_k$
	and
	$
	\sum_{v_k \in N_{\Gamma'}(v)}x_k\ge \sum_{v_k \in N_{\Gamma'}(w)}x_k$.
	
	Let $\Gamma^{''}=\Gamma'-\left\{ui \mid i\in N_{\Gamma'}(u)\right\}-\left\{wi \mid i\in N_{\Gamma'}(w)\right\}+\left\{ui \mid i\in N_{\Gamma'}(v)\right\}+\left\{wi \mid i\in N_{\Gamma'}(v)\right\}$. Obviously, $\Gamma^{''}$ is $\mathcal{K}_{r+1}^{+}$-free. We first assume that $\Gamma^{''}$ is unbalanced. By Claim \ref{At most one component of x is zero}, $x_u >0$ and $x_w>0$. By the Rayleigh principle,
	\begin{align*}
		&\lambda_1(\Gamma^{''})-\lambda_1(\Gamma')\\
		&\ge \mathbf{x}^{T}A(\Gamma^{''})\mathbf{x}-\mathbf{x}^{T}A(\Gamma')\mathbf{x}\\
		&=2x_u\sum_{v_k \in N_{\Gamma'}(v)}x_k-2x_u\sum_{v_k \in N_{\Gamma'}(u)}x_k+2x_w\sum_{v_k \in N_{\Gamma'}(v)}x_k-2x_w\sum_{v_k \in N_{\Gamma'}(w)}x_k+2x_u x_w\\
		&> 0.
	\end{align*}
	This contradicts the maximality of $\lambda_1(\Gamma^{\prime})$.
	If $\Gamma^{''}$ is balanced, this implies that the negative edge $v_1 v_2$ is a cut edge of $\Gamma^{''}$. Since $\Gamma^{''}$ is connected, suppose that \(v_1v_3\in E(\Gamma')\) or $v_2v_4\in E(\Gamma')$. Without loss of generality, \(v_1v_3\in E(\Gamma')\). Recall that $\mathbf{x}$ has at most one zero coordinate. Then we obtain that
	\[
	\lambda_{1}(\Gamma^{''}+v_2 v_3)>\lambda_{1}(\Gamma^{''})>\lambda_1(\Gamma^{\prime}).
	\]
	Since $\Gamma^{''}+v_2 v_3$ is unbalanced and $\mathcal{K}_{r+1}^{+}$-free, this contradicts the maximality of $\lambda_1(\Gamma^{\prime})$.
	
	From the above discussion, we conclude that $\tilde{\Gamma}$ is a complete $p$-partite signed graph with all positive edges. Since $\Gamma'$ is $\mathcal{K}_{r+1}^{+}$-free, we have $p\le r$.

	Suppose that $p<r$. In this time, we assert that \(v_1\) and \(v_2\) are adjacent to all vertices in $\tilde{\Gamma}$. Otherwise, assume \(v_1\) is nonadjacent to $w$ where $w\in V(\tilde{\Gamma})$. Similar to the proof of Claim \ref{connected-1}, we have $\lambda_{1}(\Gamma'+v_1 w)>\lambda_{1}(\Gamma')$ which contradicts the maximality of $\lambda_1(\Gamma^{\prime})$. Hence, \(v_1\) and \(v_2\) are adjacent to all vertices in $\tilde{\Gamma}$. Suppose that the $p$ parts of $\tilde{\Gamma}$ are $V_1,V_2,\dots,V_p$. If $|V_i|\le 1$ for any $i\in [p]$, then $n\le 2+p\le 2+r-1=r+1$. Since $n\ge r+1$, we have $n=r+1$ and $p=r-1$. Then $|V_1|=|V_2|=\dots=|V_p|=1$. Hence, $\Gamma'=T_{r}^\star(n)$. In this case, Theorem \ref{Classical spectral Turan for s-g} holds. If there exists some $i$ such that $|V_i|\ge 2$, without loss of generality, let $V_1$ be the partition set with at least two vertices in it and $u, v\in  V_1$. We can construct a complete $(p+1)$-partite graph $\Gamma_1$ with partition $u, V_1\setminus\{u\}, V_2, \dots, V_p$. Let $\Gamma_2$ be obtained by adding positive edges from \(v_1\) and \(v_2\) to all vertices in $\Gamma_1$. It is clear that $\lambda_1(\Gamma^{\prime})<\lambda_1(\Gamma_2)$. And since $p<r$, $\Gamma_2$ is also $\mathcal{K}_{r+1}^{+}$-free. This is a contradiction. This implies that \(p=r\).

	\begin{claim}\label{complete r-partite-one negative edge}
		$\Gamma'$ is obtained by adding one negative edge to a complete $r$-partite signed graph with all positive edges.
	\end{claim}
	\noindent\emph{Proof of Claim 7.}~By Claim \ref{complete r-partite}, $\Gamma'[V(\Gamma^{\prime})\setminus \left\{v_1,v_2\right\}]$ is a complete $r$-partite signed graph with all positive edges. Recall that $\Gamma'$ is $\mathcal{K}_{r+1}^{+}$-free. Hence, \(v_1\) is adjacent to vertices from at most \(r-1\) parts. By the maximality of $\lambda_{1}(\Gamma')$, we can always add positive edges so that \(v_1\) is adjacent to all vertices in \(r-1\) parts. The same holds for \(v_2\). We now prove that \(v_1\) and \(v_2\) lie in the same part. We first suppose that $\Gamma'$ is a complete $r$-partite signed graph with exactly one negative edge. Suppose that the $r$ parts of $\Gamma'$ are $V_1,V_2,\dots,V_r$, with $v_1\in V_1$ and $v_2\in V_2$. Let $\Gamma^{*}=\Gamma'-\left\{v_k v_2 \mid v_k \in N_{\Gamma'}(v_2)\setminus \left\{v_1\right\}\right\}+\left\{v_k v_2 \mid v_k \in N_{\Gamma'}(v_1)\setminus \left\{v_2\right\}\right\}$. For $\Gamma'$, we have
	\[
	\lambda_{1}(\Gamma')x_1=-x_2+\sum_{v_k \in N_{\Gamma'}(v_1)\setminus\left\{v_2\right\}}x_k,
	\]
	and
	\[
	\lambda_{1}(\Gamma')x_2=-x_1+\sum_{v_k \in N_{\Gamma'}(v_2)\setminus\left\{v_1\right\}}x_k.
	\]
	Recall that $x_2\ge x_1$. Hence,
	\[
	(\lambda_{1}(\Gamma')-1)(x_2-x_1)=\sum_{v_k \in N_{\Gamma'}(v_1)\setminus\left\{v_2\right\}}x_k-\sum_{v_k \in N_{\Gamma'}(v_2)\setminus\left\{v_1\right\}}x_k\ge 0.
	\]
	Then
	\begin{align*}
		\lambda_1(\Gamma^{*})-\lambda_1(\Gamma')&\ge \mathbf{x}^{T}A(\Gamma^{*})\mathbf{x}-\mathbf{x}^{T}A(\Gamma')\mathbf{x}\\
		&=2x_{2}\left(\sum_{v_k \in N_{\Gamma'}(v_1)\setminus\left\{v_2\right\}}x_k-\sum_{v_k \in N_{\Gamma'}(v_2)\setminus\left\{v_1\right\}}x_k\right)\\
		&\ge 0.
	\end{align*}
	If $\lambda_1(\Gamma^{*})=\lambda_1(\Gamma^{\prime})$, then $\mathbf{x}$ is an eigenvector of $A(\Gamma^{*})$ corresponding to $\lambda_1(\Gamma^{*})$ and $\sum_{v_k \in N_{\Gamma'}(v_1)\setminus\left\{v_2\right\}}x_k=\sum_{v_k \in N_{\Gamma'}(v_2)\setminus\left\{v_1\right\}}x_k$. Hence, $x_1=x_2$. Note that for $v_i\in V_1\setminus \left\{v_1\right\}$,
	\[
	\lambda_1(\Gamma')x_i=x_2+\sum_{v_k \in N_{\Gamma'}(v_i)\setminus\left\{v_2\right\}}x_k
	\]
	and
	\[
	\lambda_1(\Gamma^{*})x_i=\sum_{v_k \in N_{\Gamma'}(v_i)\setminus\left\{v_2\right\}}x_k.
	\]
	Hence, we have $x_1=x_2=0$ which contradicts Claim \ref{eigenvector of one zero coordinate}. So, \(v_1\) and \(v_2\) lie in the same part.
	
	\begin{claim}\label{add one negative edge to T_{r}(n)}
		$\Gamma'$ is obtained by adding one negative edge to $T_{r}(n)$.
	\end{claim}
	\noindent\emph{Proof of Claim 8.}~By Claim \ref{complete r-partite-one negative edge}, $\Gamma'$ is obtained by adding one negative edge to a complete $r$-partite signed graph with all positive edges. Let $V(\Gamma') = V_1 \cup \cdots \cup V_r$ be the vertex partition of $\Gamma'$ and the size of each part be $n_1,n_2,\dots,n_r$, respectively. Without loss of generality, assume $v_1,v_2\in V_1$. Since all vertices in the same part other than \(v_1\) and \(v_2\) have the same neighborhood, by the property of an eigenvector, all vertices in the same part other than \(v_1\) and \(v_2\) have the component with the same value in $\mathbf{x}$. If there exists a pair of parts with the difference in sizes at least $2$, without loss of generality, let $n_i-n_j\ge 2$.
	
	\noindent\emph{Case~1.}~$i,j\ne 1$. For $v_i\in V_i$ and $v_j\in V_j$, by Claim \ref{At most one component of x is zero}, $x_i>0$ and $x_j>0$. We have
	\[
	\lambda_{1}(\Gamma')x_i=n_j x_j+B,~~\lambda_{1}(\Gamma')x_j=n_i x_i+B
	\]
	where $B=\sum_{w\in V(\Gamma')\setminus \left(V_i\cup V_j\right)}x_w$. By direct calculation,
	\begin{equation}\label{x_j-x_i}
		(\lambda_{1}(\Gamma')+n_i)x_i=(\lambda_{1}(\Gamma')+n_j)x_j.
	\end{equation}
	Since $n_i>n_j$, we have $x_i<x_j$. Hence,
	\[
	\lambda_{1}(\Gamma')=n_j \frac{x_j}{x_i}+\frac{B}{x_i}>n_j.
	\]
	Let $\Gamma^{*}=\Gamma'-\left\{v_i v_k \mid v_k\in N_{\Gamma'}(v_i)\right\}+\left\{v_i v_k \mid v_k\in N_{\Gamma'}(v_j)\setminus \left\{v_i\right\}\right\}$. Obviously, $\Gamma^{*}$ is unbalanced and $\mathcal{K}_{r+1}^{+}$-free. We have
	\[
	\lambda_1(\Gamma^{*})-\lambda_1(\Gamma')\ge 2x_i\left[(n_i-1)x_i-n_j x_j\right].
	\]
	However, by \eqref{x_j-x_i},
	\begin{align*}
		&\left[(n_i-1)x_i-n_j x_j\right]\left[(\lambda_{1}(\Gamma')+n_i)x_i\right]\\
		&=(n_i-1)x_i(\lambda_{1}(\Gamma')+n_j)x_j-n_j x_j(\lambda_{1}(\Gamma')+n_i)x_i\\
		&=x_i x_j\left[(n_i-n_j-1)\lambda_{1}(\Gamma')-n_j\right]\\
		&\ge x_i x_j(\lambda_{1}(\Gamma')-n_j)\\
		&>0.
	\end{align*}
	Hence, $(n_i-1)x_i-n_j x_j>0$ which implies that $\lambda_1(\Gamma^{*})>\lambda_1(\Gamma')$. This contradicts the maximality of $\lambda_1(\Gamma^{\prime})$.
	
	\noindent\emph{Case~2.}~$i=1$, $j\ne 1$. Since $n_1-n_j\ge 2$ and $n_j\ge 1$, we have $n_1\ge 3$. For $v_l\in V_1 \setminus \left\{v_1,v_2\right\}$ and $v_j\in V_j$, by Claim \ref{At most one component of x is zero}, $x_l>0$ and $x_j>0$. By symmetry, $x_1=x_2$. Note that
	\[
	\lambda_{1}(\Gamma')x_1=-x_2+\sum_{v_k \in N_{\Gamma}(v_1)\setminus\left\{v_2\right\}}x_k
	\]
	and
	\[
	\lambda_{1}(\Gamma')x_l=\sum_{v_k \in N_{\Gamma}(v_1)\setminus\left\{v_2\right\}}x_k.
	\]
	Hence, $x_1=x_2=\frac{\lambda_{1}(\Gamma')}{\lambda_{1}(\Gamma')+1}x_l$. Note that
	\[
	\lambda_{1}(\Gamma')x_l=n_j x_j+B_1,~~\lambda_{1}(\Gamma')x_j=(n_1-2) x_l+x_1+x_2+B_1
	\]
	where $B_1=\sum_{w\in V(\Gamma')\setminus \left(V_1\cup V_j\right)}x_w$. By direct calculation,
	\begin{equation}\label{x_j-x_i-1}
		\left(\lambda_{1}(\Gamma')+n_1-2+\frac{2\lambda_{1}(\Gamma')}{\lambda_{1}(\Gamma')+1}\right)x_l=(\lambda_{1}(\Gamma')+n_j)x_j.
	\end{equation}
	Let $\Gamma^{**}=\Gamma'-\left\{v_l v_k \mid v_k\in N_{\Gamma'}(v_l)\right\}+\left\{v_l v_k \mid v_k\in N_{\Gamma'}(v_j)\setminus \left\{v_l\right\}\right\}$. Obviously, $\Gamma^{**}$ is unbalanced and $\mathcal{K}_{r+1}^{+}$-free. We have
	\[
	\lambda_1(\Gamma^{*})-\lambda_1(\Gamma')\ge 2x_l\left[\left(n_1-2+\frac{2\lambda_{1}(\Gamma')}{\lambda_{1}(\Gamma')+1}\right)x_l-n_j x_j\right].
	\]
	However, by \eqref{x_j-x_i-1},
	\begin{align*}
		&\left[\left(n_1-2+\frac{2\lambda_{1}(\Gamma')}{\lambda_{1}(\Gamma')+1}\right)x_l-n_j x_j\right](\lambda_{1}(\Gamma')+n_j)x_j\\
		&=x_l x_j\left(n_1-2+\frac{2\lambda_{1}(\Gamma')}{\lambda_{1}(\Gamma')+1}\right)(\lambda_{1}(\Gamma')+n_j)-n_jx_l x_j\left(\lambda_{1}(\Gamma')+n_1-2+\frac{2\lambda_{1}(\Gamma')}{\lambda_{1}(\Gamma')+1}\right)\\
		&=\lambda_{1}(\Gamma')x_l x_j\left(n_1-n_j-2+\frac{2\lambda_{1}(\Gamma')}{\lambda_{1}(\Gamma')+1}\right)\\
		&\ge 2x_l x_j\frac{\lambda^2_{1}(\Gamma')}{\lambda_{1}(\Gamma')+1}\\
		&>0.
	\end{align*}
	This implies $\lambda_1(\Gamma^{**})>\lambda_1(\Gamma')$ which contradicts the maximality of $\lambda_1(\Gamma^{\prime})$. Hence, $n_1=2$.
	
	\noindent\emph{Case~3.}~$i\ne 1$, $j=1$ and $n_1\ge 3$. For $v_l\in V_1 \setminus \left\{v_1,v_2\right\}$ and $v_i\in V_i$, by Claim \ref{At most one component of x is zero}, $x_l>0$ and $x_i>0$. Similar to Case $2$, we have
	$x_1=x_2=\frac{\lambda_{1}(\Gamma')}{\lambda_{1}(\Gamma')+1}x_l$,
	\begin{equation}\label{x_l-x_i}
		\lambda_{1}(\Gamma')x_l=n_i x_i+B_2,~~\lambda_{1}(\Gamma')x_i=(n_1-2) x_l+x_1+x_2+B_2
	\end{equation}
	where $B_2=\sum_{w\in V(\Gamma')\setminus \left(V_1\cup V_i\right)}x_w$. By direct calculation,
	\begin{equation}\label{x_j-x_i-2}
		\left(\lambda_{1}(\Gamma')+n_1-2+\frac{2\lambda_{1}(\Gamma')}{\lambda_{1}(\Gamma')+1}\right)x_l=(\lambda_{1}(\Gamma')+n_i)x_i.
	\end{equation}
	Let $\Gamma^{''}=\Gamma'-\left\{v_i v_k \mid v_k\in N_{\Gamma'}(v_i)\right\}+\left\{v_i v_k \mid v_k\in N_{\Gamma'}(v_l)\setminus \left\{v_i\right\}\right\}$. Obviously, $\Gamma^{''}$ is unbalanced and $\mathcal{K}_{r+1}^{+}$-free. We have
	\[
	\lambda_1(\Gamma^{''})-\lambda_1(\Gamma')\ge 2x_i\left[(n_i-1) x_i-\left(n_1-2+\frac{2\lambda_{1}(\Gamma')}{\lambda_{1}(\Gamma')+1}\right)x_l\right].
	\]
	However, by \eqref{x_j-x_i-2},
	\begin{align*}
		&\left[(n_i-1) x_i-\left(n_1-2+\frac{2\lambda_{1}(\Gamma')}{\lambda_{1}(\Gamma')+1}\right)x_l\right](\lambda_{1}(\Gamma')+n_i)x_i\\
		&=x_i x_l(n_i-1)\left(\lambda_{1}(\Gamma')+n_1-2+\frac{2\lambda_{1}(\Gamma')}{\lambda_{1}(\Gamma')+1}\right)-x_i x_l\left(n_1-2+\frac{2\lambda_{1}(\Gamma')}{\lambda_{1}(\Gamma')+1}\right)(\lambda_{1}(\Gamma')+n_i)\\
		&=x_i x_j\left[\left(n_i-n_1+2+\frac{2\lambda_{1}(\Gamma')}{\lambda_{1}(\Gamma')+1}\right)\lambda_{1}(\Gamma')-\left(n_1-2+\frac{2\lambda_{1}(\Gamma')}{\lambda_{1}(\Gamma')+1}\right)\right]\\
		&\ge x_i x_j\left[\left(4+\frac{2\lambda_{1}(\Gamma')}{\lambda_{1}(\Gamma')+1}\right)\lambda_{1}(\Gamma')-\left(n_1-2+\frac{2\lambda_{1}(\Gamma')}{\lambda_{1}(\Gamma')+1}\right)\right].
	\end{align*}
	
	If $x_l\ge x_i$, by \eqref{x_l-x_i}, then
	\[
	\lambda_{1}(\Gamma')=\left(n_1-2+\frac{2\lambda_{1}(\Gamma')}{\lambda_{1}(\Gamma')+1}\right)\frac{x_l}{x_i}+\frac{B_2}{x_i}\ge n_1-2+\frac{2\lambda_{1}(\Gamma')}{\lambda_{1}(\Gamma')+1}.
	\]
	Clearly, $\left(4+\frac{2\lambda_{1}(\Gamma')}{\lambda_{1}(\Gamma')+1}\right)\lambda_{1}(\Gamma')>n_1-2+\frac{2\lambda_{1}(\Gamma')}{\lambda_{1}(\Gamma')+1}$.
	
	If $x_l< x_i$, by \eqref{x_l-x_i}, then
	\[
	\lambda_{1}(\Gamma')=n_i \frac{x_i}{x_l}+\frac{B_2}{x_l}>n_i.
	\]
	We have
	\[
	\left(4+\frac{2\lambda_{1}(\Gamma')}{\lambda_{1}(\Gamma')+1}\right)\lambda_{1}(\Gamma')>n_i\left(4+\frac{2\lambda_{1}(\Gamma')}{\lambda_{1}(\Gamma')+1}\right)>n_1-2+\frac{2\lambda_{1}(\Gamma')}{\lambda_{1}(\Gamma')+1}.
	\]
	Therefore, we all have
	\[
	\left(4+\frac{2\lambda_{1}(\Gamma')}{\lambda_{1}(\Gamma')+1}\right)\lambda_{1}(\Gamma')>n_1-2+\frac{2\lambda_{1}(\Gamma')}{\lambda_{1}(\Gamma')+1}.
	\]
	This implies $\lambda_1(\Gamma^{''})>\lambda_1(\Gamma')$ which contradicts the maximality of $\lambda_1(\Gamma^{\prime})$.
	
	\noindent\emph{Case~4.}~$i\ne 1$, $j=1$ and $n_1=2$. Similar to Case $3$, for $v_i\in V_i$, we have
	\begin{equation}\label{x_i-x_1}
		\lambda_{1}(\Gamma')x_1=-x_2+n_i x_i+B_2,~~\lambda_{1}(\Gamma')x_i=x_1+x_2+B_2
	\end{equation}
	where $B_2=\sum_{w\in V(\Gamma')\setminus \left(V_1\cup V_i\right)}x_w$. By direct calculation,
	\begin{equation}\label{x_i-x_1-2}
		(\lambda_{1}(\Gamma')+3)x_1=(\lambda_{1}(\Gamma')+n_i)x_i.
	\end{equation}
	Note that $n_i-n_j\ge 2$. Then $n_i\ge 4$. By \eqref{x_i-x_1-2}, we have $x_1>x_i$. Let $\dot{\Gamma}=\Gamma'-\left\{v_i v_k \mid v_k\in N_{\Gamma'}(v_i)\right\}+\left\{v_i v_k \mid v_k\in N_{\Gamma'}(v_1)\setminus \left\{v_i\right\}\right\}$. Obviously, $\dot{\Gamma}$ is unbalanced and $\mathcal{K}_{r+1}^{+}$-free. We have
	\[
	\lambda_1(\dot{\Gamma})-\lambda_1(\Gamma')\ge 2x_i\left[(n_i-1) x_i-2x_1\right].
	\]
	However, by \eqref{x_i-x_1-2},
	\begin{align*}
		&\left[(n_i-1) x_i-2x_1\right](\lambda_{1}(\Gamma')+n_i)x_i\\
		&=x_1 x_i(\lambda_{1}(\Gamma')+3)(n_i-1)-2x_1 x_i(\lambda_{1}(\Gamma')+n_i)\\
		&=x_1 x_i\left[(\lambda_{1}(\Gamma')+3)(n_i-1)-2(\lambda_{1}(\Gamma')+n_i)\right]\\
		&=x_1 x_i\left[(n_i-3)\lambda_{1}(\Gamma')+n_i-3\right]\\
		&>x_1 x_i(\lambda_{1}(\Gamma')+1)\\
		&>0.
	\end{align*}
	This implies $\lambda_1(\dot{\Gamma})>\lambda_1(\Gamma')$, which contradicts the maximality of $\lambda_1(\Gamma^{\prime})$.
	
	From the four cases above, we conclude that the only possibility is $n_1=2$, $n_j=0$, and each of the remaining $r-1$ parts has size at most $1$. Hence, $n\le r$ which contradicts $n\ge r+1$. From the above discussion, we obtain $|n_i-n_j|\le 1$ for any $i,j\in [r]$.

	By Claim \ref{add one negative edge to T_{r}(n)}, we only need to consider whether the unique negative edge $v_1 v_2$ lies in a partition of size \(\lceil n/r\rceil\) or a partition of size \(\lfloor n/r\rfloor\). If $q=1$, then $n=r+s$. The negative edges can only lie within the partition class of size $2$, so we already obtain $\Gamma'=T_r^{*}(n)$.  If $q\ge 2$ and $s=0$, we have $\Gamma'= T_{r}^\star(n)=T_{r}^\diamond(n)$. If $q\ge 2$ and $1\le s<r$, by Theorem \ref{eigenvalue-T_{r}^{+-}(n)}, $\lambda_{1}(T_{r}^\diamond(n))<\lambda_{1}(T_{r}^\star(n))$. By the maximality of $\lambda_1(\Gamma^{\prime})$, $\Gamma'= T_{r}^\star(n)$.

	This completes the proof.\qed
\end{Tproof}

\begin{prop}\label{the bound of Tr(n)}
	For every $r\ge 2$ and $n\ge r + 1$,
	\[
	\lambda_{1}(T_{r}^\star(n))<\frac{2e(T_r(n))}{n}.
	\]
\end{prop}	

\begin{proof} Set $$a=\frac{2e(T_r(n))}{n},~~c=s(q+1)$$
	If $s=0$, then $T_r(n)$ is $a$-regular, so $\Phi_T(a)=1$. Recall that
	$$
	\Phi_k(t) = \Phi_T(t) - h_k(t) + g_k(t)
	$$
	and
	$$
	\Delta_k(a)= h_k(t) - g_k(t)= \frac{2t}{(t+1)(t+k)\left(t + k - \dfrac{2}{t+1}\right)}>0
	$$
	We get $\Delta_q(a)<1$. Since $\Delta_q$ is strictly decreasing and $\Phi_q(\lambda_{1}(T_{r}^\star(n)))=1$, we obtain $\lambda_{1}(T_{r}^\star(n))<a$.
	
	Now suppose that $s>0$. There are $n - c$ vertices in parts of size $q$ and $c$ vertices in parts of size $q+1$. Since $2e(T_r(n)) = n^2 - nq - c$,
	we have
	$$
	a = n - q - \frac{c}{n}.
	$$
	Therefore
	\begin{equation}\label{eq-1}
		\Phi_T(a) - 1 = \frac{n-c}{a+q} + \frac{c}{a+q+1} - 1 = \frac{c(n-c)}{(n^2 - c)(n^2 + n - c)}.
	\end{equation}
	Since $0 < c < n$, we have
	\[
	c(n-c) \leq \frac{n^2}{4},~~n^2 - c > n(n-1),~~n^2 + n - c > n^2.
	\]
	Hence,
	\begin{equation}\label{eq-2}
		\Phi_T(a) - 1 < \frac{1}{4n(n-1)}.
	\end{equation}
	Clearly, $e(T_r(n))\ge e(T_2(n))=\lfloor \frac{n^2}{4} \rfloor$. Then
	\begin{equation}\label{eq-3}
		a=\frac{2e(T_r(n))}{n}\ge \frac{n^2-1}{2n}.
	\end{equation}
	Since $a + q + 1 = n + 1 - \frac{c}{n} < n + 1$ and $a + 1 < n + 1$, by \eqref{eq-3}, we have
	\begin{equation}\label{eq-4}
		\Delta_{q+1}(a) = \frac{2a}{(a + q + 1)\big((a + 1)(a + q + 1) - 2\big)} > \frac{n - 1}{n(n + 1)^2}.
	\end{equation}
	By \eqref{eq-2}, \eqref{eq-3} and \eqref{eq-4}, we have
	\[
	\Phi_{q+1}(a)=\Phi_T(a) -\Delta_{q+1}(a)<1+\frac{1}{4n(n-1)}-\frac{n - 1}{n(n + 1)^2}.
	\]
	For $n \ge 4$,
	\[
	\frac{n - 1}{n(n + 1)^2} > \frac{1}{4n(n - 1)}.
	\]
	By Lemma \ref{eigenvalue of A(Sigma)}, $\Phi_k$ is strictly decreasing. Hence, we obtain $\lambda_{1}(T_{r}^\star(n))< a$.
	
	For $n=3$ and $r=2$, we have
	\[
	\lambda_1(T_2^*(3))=1<\frac{4}{3}=a.
	\]
	Thus, for every admissible $r,n$,
	$$
	\lambda_1(T_r^\star(n)) < \frac{2t_r(n)}{n}.
	$$
	
	This completes the proof.\qed
\end{proof}

\begin{remark}
	{\em (a)~Let $\Gamma$ be an unbalanced signed graph of order $n$ with $\omega_{b}(\Gamma)=\omega_{b}$. By Theorems \ref{Classical spectral Turan for s-g} and \ref{eigenvalue-T_{r}^{+-}(n)}, we have
		\[
		\lambda_{1}(\Gamma)\le \lambda_{1}(T_{\omega_{b}}^\star(n))< \lambda_{1}(T_{\omega_{b}}(n))\le \left(1-\frac{1}{\omega_{b}}\right)n.
		\]
		This implies that Theorem \ref{Classical spectral Turan for s-g}  yields sharper bounds than Lemma \ref{Wang-Yan-Qian-LAA-2021} when $\Gamma$ is unbalanced.
		
		(b)~Let $\Gamma$ be an unbalanced $\mathcal{K}^{+}_{r+1}$-free signed graph. By the Rayleigh principle, we have
		\[
		\frac{2(|E^+(\Gamma)|-|E^-(\Gamma)|)}{n}\le \lambda_1(\Gamma).
		\]
		By the definition of $S(\Gamma)$, we obtain $S(\Gamma)\le \frac{n}{2}\lambda_1(\Gamma)$. By Proposition \ref{the bound of Tr(n)} and Theorem \ref{Classical spectral Turan for s-g}, we have
		\[
		S(\Gamma)\le \frac{n}{2}\lambda_1(\Gamma)\le \frac{n}{2}\lambda_1(T_{r}^\star(n))<e(T_r(n)).
		\]
		So, $S(\Gamma)\le e(T_r(n))-1=S(T_{r}^\star(n))$. Thus, the spectral bound of Theorem \ref{Classical spectral Turan for s-g} implies the edge bound given by Theorem \ref{signed Turan theorem}.

		(c)~Theorem \ref{Classical spectral Turan for s-g} can not imply the extremal-graph characterization of Theorem \ref{signed Turan theorem}, since there exists an $n$-vertex unbalanced signed graph \(\Gamma\) such that $\Gamma \not \sim T_{r}^\diamond(n)$, \(\lambda_{1}(\Gamma)< \lambda_{1}(T_{r}^\star(n))\) and \(S(\Gamma)\ge S(T_{r}^{\star}(n))\). For example, let $r=2$ and $n=6$, and suppose that $\Gamma$ has the adjacency matrix
		\[
		A(\Gamma)=
		\begin{pmatrix}
			0 & 1 & 1 & 0 & 1 & 0 \\
			1 & 0 & 1 & 1 & 0 & 0 \\
			1 & 1 & 0 & -1 & 0 & 1 \\
			0 & 1 & -1 & 0 & 1 & 1 \\
			1 & 0 & 0 & 1 & 0 & 1 \\
			0 & 0 & 1 & 1 & 1 & 0
		\end{pmatrix}.
		\]
		Then $\Gamma$ is unbalanced, $\Gamma \not \sim T_{2}^\diamond(6)$ and $\epsilon(\Gamma)=1$. Hence,
		\[
		8=S(\Gamma)=S(T_2^\star(6)).
		\]
		By simple computation, we obtain that the characteristic polynomial of $\Gamma$ is
		\[
		\chi_\Gamma(x)=(x^3-7x-1)(x^3-3x+1),
		\]
		and the characteristic polynomial of $T_2^\star(6)$ is
		\[
		\chi_{T_2^\star(6)}(x)=x^{2}(x-1)(x^{3}+x^{2}-9x-3).
		\]
		Let $f(x)=x^{3}-7x-1~\text{and}~g(x)=x^{3}+x^{2}-9x-3$.
		Clearly, \(\lambda_{1}(\Gamma)\) and \(\lambda_{1}(T_{2}^\star(6))\) are the largest zeros of \(f(x)\) and \(g(x)\), respectively. Since
		\[
		f(1+\sqrt{3}) = 2-\sqrt{3} > 0~,~f(2)<0,
		\]
		and $f$ is strictly increasing on $[2,\infty)$, we have $2 < \lambda_{1}(\Gamma) < 1+\sqrt{3}$. Using $f(\lambda_{1}(\Gamma))=0$, we have
		\[
		g(\lambda_{1}(\Gamma)) = \lambda_{1}(\Gamma)^2 - 2\lambda_{1}(\Gamma) - 2 = (\lambda_{1}(\Gamma)-1)^2 - 3 < 0.
		\]
		The polynomial $g$ is strictly increasing on $[2,\infty)$, with $g(2) < 0$ and $g(3) > 0$. This implies $\lambda_1(\Gamma) < \lambda_1(T_2^\star(6))$.
		
	}
	
\end{remark}

Let $r=2$. Theorem \ref{Classical spectral Turan for s-g} implies the following result immediately.
\begin{cor}\label{balanced triangle}
	Let $\Gamma=(G,\sigma)$ be an unbalanced signed graph of order $n~(\ge 3)$. If $\Gamma$ is $\mathcal{K}^{+}_{3}$-free, then
	\[
	\lambda_{1}(\Gamma)\le \lambda_{1}(T_{2}^\star(n)),
	\]
	with equality holding if and only if $\Gamma \sim T_{2}^\star(n)$.
\end{cor}

Noting that the underlying graph of the extremal graph in Corollary \ref{balanced triangle} is non-bipartite, we conclude that Corollary \ref{balanced triangle} is a signed version of \cite[Theorem~1.4]{Lin-Ning-Wu-CPC-2021}. The following corollary gives the result obtained by Brunetti and Stani{\'c}.
\begin{cor}\emph{(See \cite{Brunetti-Stanic-CAM-2022})}
	Every unbalanced signed graph $\Gamma$ on $n \ge 3$ vertices satisfies
	\[
	\lambda_1(\Gamma) \le \beta_n := \frac{n - 4 + \sqrt{n^2 + 4n - 12}}{2}.
	\]
	Equality holds if and only if $\Gamma$ is switching equivalent to the complete signed graph having exactly one negative edge.
\end{cor}
\begin{proof}
	Note that $\Gamma$ is $\mathcal{K}^{+}_{n}$-free. Applying Theorem \ref{Classical spectral Turan for s-g} with $r = n - 1$, we have
	\[
	\lambda_{1}(\Gamma)\le \lambda_{1}(T_{n-1}^\star(n)).
	\]
	The graph $T_{n-1}^*(n)$ is the complete signed graph having exactly one negative edge.  Its adjacency matrix has quotient
	\[
	Q = \begin{pmatrix}
		-1 & n - 2 \\
		2 & n - 3
	\end{pmatrix}.
	\]
	The larger root of
	\[
	\det(xI - Q) = x^2 - (n - 4)x - (3n - 7)
	\]
	is $\beta_n$. It is the largest eigenvalue: the antisymmetric vector on the negative-edge endpoints has eigenvalue $1$, vectors summing to zero on the other $n - 2$ vertices have eigenvalue $-1$, and $\beta_n \ge 1$, with equality only at $n = 3$. Thus $\lambda_{1}(T_{n-1}^\star(n)) = \beta_n$, and both the inequality and equality statement follow from Theorem \ref{Classical spectral Turan for s-g}.
	\qed\end{proof}

\subsection{Proof~of~Theorem~\ref{Classical spectral Turan for s-g-2}}

\begin{Tproof}\textbf{~of~Theorem~\ref{Classical spectral Turan for s-g-2}.}~Suppose that $s\ne 0$. In this case, Claims \ref{eigenvector of one zero coordinate}--\ref{add one negative edge to T_{r}(n)} in Theorem \ref{Classical spectral Turan for s-g} all hold. If $q=1$, then $n=r+s$. The negative edges can only lie within the partition class of size $2$, so we already obtain $\Gamma'=T_r^{*}(n)$, a contradiction. If $q\ge 2$ and $1\le s<r$, by Theorem \ref{eigenvalue-T_{r}^{+-}(n)}, $\lambda_{1}(T_{r}^\diamond(n))<\lambda_{1}(T_{r}^\star(n))$. Since $\Gamma^{\prime} \not \sim T_{r}^\star(n)$, we have $\Gamma'= T_{r}^\diamond(n)$.
	
	In the following, we consider $s=0$. Let $\Gamma=(G,\sigma)$ be a signed graph having the maximum index over all $\mathcal{K}^{+}_{r+1}$-free unbalanced signed graph of order $n$ where $\Gamma \not \sim  T_{r}^\star(n)$. By Lemma \ref{Sun-Liu-Lan-LAA-2022}, let $\Gamma^{\prime}=(G,\sigma^{\prime})$ be the signed graph switching equivalent to $\Gamma$ for which $\lambda_{1}(\Gamma^{\prime})$ admits a non-negative eigenvector. Let $V(\Gamma^{\prime}) = \left \{v_1,v_2,\dots,v_n \right \}$ and $\mathbf{x}=\left (x_1,x_2,\dots,x_n \right )^T$ be the non-negative unit eigenvector of $A(\Gamma^{\prime})$ corresponding to $\lambda_1(\Gamma^{\prime})$. By Lemmas \ref{Zaslavsky-DAM-1982} and \ref{Hou-Tang-Wang-AMC-2019}, $\Gamma^{\prime}=(G,\sigma^{\prime})$ is also $\mathcal{K}^{+}_{r+1}$-free unbalanced signed graph with the maximum index and $\Gamma \not \sim  T_{r}^\star(n)$. Note that $\Gamma_{n}(r)$ is $\mathcal{K}^{+}_{r+1}$-free unbalanced signed graph and $\Gamma_{n}(r)\not \sim  T_{r}^\star(n)$. By the Rayleigh principle, we have
	\begin{equation}\label{Index of signed Turan graph-2}
		\lambda_{1}(\Gamma^{\prime})\ge \lambda_1(\Gamma_{n}(r))\ge \frac{\mathbf{1}^T A(\Gamma_{n}(r))\mathbf{1}}{\mathbf{1}^T \mathbf{1}}=L- \frac{s(r - s)}{rn} - \frac{4}{n}.
	\end{equation}
	\begin{claim}\label{eigenvector of one zero coordinate-1}
		The eigenvector $\mathbf{x}$ has at most one zero coordinate.
	\end{claim}
	\noindent\emph{Proof of Claim 9.}~If $\mathbf{x}$ has two zero coordinates, say $x_i$ and $x_j$. By Lemma \ref{Wang-Yan-Qian-LAA-2021}, then
	$$
	\lambda_1(\Gamma^{\prime})=\mathbf{x}^{T}A(\Gamma^{\prime})\mathbf{x}\le \lambda_1(\Gamma^{\prime}-v_i-v_j)\le \left(1-\frac{1}{\omega_b(\Gamma)}\right)(n-2)\le \left(1-\frac{1}{r}\right)(n-2).
	$$
	If $q\ge 2$, then $n\ge 2r+s$, so
	\[
	2(r-1)n - s(r-s)-4r \ge (4r+s)(r-2)+s^2.
	\]
	Clearly, $(4r+s)(r-2)+s^2=0$ if and only if $r=2$ and $s=0$. Hence,
	\begin{align*}
		&L- \frac{s(r - s)}{rn} - \frac{4}{n}-\left(1-\frac{1}{r}\right)(n-2)
		\\
		&= \frac{2(r-1)n - s(r-s)-4r}{rn}\\
		&\ge 0,
	\end{align*}
	with equality holding if and only if $r=2$ and $s=0$. We suppose that $r=2$ and $s=0$. Then
	\[
	\lambda_1(\Gamma^{\prime})=\lambda_1(\Gamma_{n}(r))=\frac{\mathbf{1}^T A(\Gamma_{n}(r))\mathbf{1}}{\mathbf{1}^T \mathbf{1}}.
	\]
	This implies $\lambda_1(\Gamma_{n}(r))\mathbf{1}=A(\Gamma_{n}(r))\mathbf{1}$, a contradiction. Then
	\[
	L- \frac{s(r - s)}{rn} - \frac{4}{n}>\left(1-\frac{1}{r}\right)(n-2),
	\]
	which contradicts equality \eqref{Index of signed Turan graph-2}.
	
	If $q=1$, then $n=r+s$ with $s\ge 1$, and
	\[
	2(r-1)n - s(r-s)-4r = 2r(r-3)+s(r-2)+s^2.
	\]
	Clearly, if $r\ge 3$, then $2(r-1)n - s(r-s)-4r>0$. If $r=2$, we have $s=1$ and then $n=3$. In this case, $\Gamma'\sim T_{2}^\star(3)$, a contradiction.
	Hence,
	\begin{align*}
		&L- \frac{s(r - s)}{rn} - \frac{2}{n}-\left(1-\frac{1}{r}\right)(n-2)
		\\
		&= \frac{2(r-1)n - s(r-s)-4r}{rn}\\
		&>0,
	\end{align*}
	which contradicts equality \eqref{Index of signed Turan graph-2}.

	Now, Claims \ref{eigenvector of one zero coordinate}--\ref{complete r-partite} in Theorem \ref{Classical spectral Turan for s-g} all hold. If \(p<r\), then $\Gamma'=T_{r}^\star(n)$, a contradiction. We now consider the case \(p=r\). By Claim \ref{complete r-partite}, $\Gamma'[V(\Gamma^{\prime})\setminus \left\{v_1,v_2\right\}]$ is a complete $r$-partite signed graph with all positive edges. Recall that $\Gamma'$ is $\mathcal{K}_{r+1}^{+}$-free. Hence, \(v_1\) is adjacent to vertices from at most \(r-1\) parts. By the maximality of $\lambda_{1}(\Gamma')$, we can always add positive edges so that \(v_1\) is adjacent to all vertices in \(r-1\) parts. The same holds for \(v_2\). We divide our discussion into the following two cases.
	
	\noindent\emph{Case~1.}~\(v_1\) and \(v_2\) lie in the same part. At this point, Claims \ref{complete r-partite-one negative edge} and \ref{add one negative edge to T_{r}(n)} hold. If $q=1$, then $n=r+s$. The negative edges can only lie within the partition class of size $2$, so we already obtain $\Gamma'=T_r^{*}(n)$, a contradiction. If $q\ge 2$ and $s=0$, we have $\Gamma'= T_{r}^\star(n)=T_{r}^\diamond(n)$, a contradiction.
	
	\noindent\emph{Case~2.}~\(v_1\) and \(v_2\) do not lie in the same part. At this point, $\Gamma'$ is a complete $r$-partite signed graph with exactly one negative edge. By Lemma \ref{Cai-Zhou-ar-2026}, we have $\Gamma'=\Gamma_{n}(r)$.
	
	This completes the proof.\qed
	
\end{Tproof}

\section*{Statements and Declarations}

\noindent \textbf{Competing interests}~~No potential competing interest was reported by the authors.

\medskip

\noindent \textbf{Data availability statements}~~Data sharing not applicable to this article as no datasets were generated or analysed during the current study.

\end{document}